\documentclass[12pt]{amsart}
\usepackage{amssymb}
\usepackage{graphicx, color,amsmath}
\usepackage{subfigure}
\usepackage[linecolor=white,backgroundcolor=white,bordercolor=white,textsize=tiny]{todonotes}
\usepackage{bm}
\graphicspath{{./Figures/}}
\allowdisplaybreaks

\newtheorem{theorem}{Theorem}
\newtheorem{lemma}[theorem]{Lemma}
\newtheorem{proposition}[theorem]{Proposition}

\newtheorem{hypothesis}{Hypothesis}

\theoremstyle{definition}
\newtheorem{definition}[theorem]{Definition}

\theoremstyle{remark}
\newtheorem{remark}{Remark}
\newtheorem{example}[theorem]{Example}

\newcommand{\E}{\mathrm{E}}

\newcommand{\Tr}{\mathrm{Tr}}

\newcommand{\ab}{\allowbreak}

\let\phi=\varphi

\newcounter{jmpnumber}

\makeatletter
\newcounter{int}
\newcommand{\listcomments}{%
\smallskip\hrule\smallskip\noindent 
\@whilenum\value{int}<\thejmpnumber\do
{\tiny Comm. \theint{} is on page \pageref{jmp\theint}\stepcounter{int}, }
\smallskip\hrule\smallskip}
\makeatother	
\newcommand{\etc}{, \dots,}
\newcommand\eq{\begin{eqnarray*}}
\newcommand\qe{\end{eqnarray*}}
\newcommand\eqa{\begin{eqnarray}}
\newcommand\qea{\end{eqnarray}}
\newcommand\DT{{\mcal{DT}}}
\newcommand\DUFT{{\mcal{DU}}\&{\mcal{FT}}}
\def\mbf{\bm}
\def\mcal{\mathcal}
\def\mbb{\mathbb}
\def\mrm{\mathrm}
\def\limN{\underset{N \rightarrow \infty}\longrightarrow}
\def\Nlim{\underset{N \rightarrow \infty}\lim}

\def\ybox{\dimen0=\hsize \dimen1=30pt%
\divide\dimen0 by \dimen1
\count100=\dimen0 \dimen0=\hsize
\divide\dimen0 by \count100%
\dimen1=\dimen0\divide\dimen1 by 2
{\parindent0pt\hbox to
\hsize{\cleaders\vbox{\hsize\dimen0
\hrule\vrule  height1.2pt%
\advance\dimen1 by -0.8pt\hskip
\dimen1\advance\dimen1 by 1.2pt \vrule
width\dimen1\hrule}\hfill}}}

\newcommand{\thebottomline}{%
\renewcommand{\thefootnote}{}
\renewcommand{\footnoterule}{}
\phantom{M}\footnotetext{\hfill\tiny\textit{\noindent\romannumeral\day.\romannumeral\month.\romannumeral\year}}}

\title[{Joint Global Fluctuations}] {Joint Global Fluctuations of complex Wigner and deterministic Matrices}

\author[male]{Camile Male}
\address{CNRS and 
Institut de Math\'ematiques
de Bordeaux, Universit\'e de Bordeaux,
33400 Talence,  France}
\email{camille.male@math.u-bordeaux.fr}

\author[mingo]{James A. Mingo$^{(*)}$} 
\address{Department
  of Mathematics and Statistics, Queen's University, Jeffery
  Hall, Kingston, Ontario, K7L 3N6, Canada}
\email{mingo@mast.queensu.ca}

\author[peche]{Sandrine P\'ech\'e$^{(**)}$} \address{Laboratoire de Probabilit{\'e}s et Mod{\`e}les Al{\'e}atoires, Universit{\'e} de Paris, Paris, France}
\email{peche@math.univ-paris-diderot.fr }

\author[speicher]{Roland Speicher$^{(***)}$} \address{Saarland University, 
Department of Mathematics, 
66041 Saar\-br\"ucken, 
Germany}
\email{speicher@math.uni-sb.de} 

\thanks{$^*$ Research supported by a Discovery Grant from
  the Natural Sciences and Engineering Research Council of
  Canada}
\thanks{$^{**}$ Research was supported by the Institut Universitaire de France.}
\thanks{$^{***}$ Research supported by the SFB-TRR 195}
\begin{document}

\begin{abstract}We characterize the limiting fluctuations of traces of several independent Wigner matrices and deterministic matrices under mild conditions. A CLT holds but in general the families are not asymptotically free of second order and the limiting covariance depends on more information on the deterministic matrices than their limiting $^*$-distribution.
\end{abstract}

\maketitle
\section{Introduction}

A Wigner matrix is a $N \times N$ Hermitian random matrix
of the form $X_N = \frac{1}{\sqrt N}(x_{ij})$, where:
\begin{itemize}
\item
the sub-diagonal entries $(x_{ij})_{i\leq j}$ are centered
independent complex random variables, (which can also be real-valued, if their imaginary part vanishes identically),
\item
the diagonal entries are identically distributed real random variables, 
\item the distribution of $x_{ij}$ does not depend on $N$ and it has bounded 
moments of all orders ($\E(|x_{ij}|^k) < \infty$ for all $i, j, k$).
\end{itemize}

Thanks to the seminal work of Wigner \cite{Wig51}, we know that the empirical 
spectral distribution of a Wigner matrix converges in moments to the semicircular distribution, 
namely
		$$\E\Big[ \frac 1 N \Tr X_N^k \Big] \limN \int_{-2\sigma}^{2\sigma} t^k \frac{\sqrt{4\sigma^2-t^2}}{2\pi\sigma^2 } \mrm dt, \ \forall k \geq 1,$$
where $\sigma^2 =\E[|x_{i,j}|^2]$ for $i\neq j$. The convergence holds also almost surely and in probability. It also holds in weak-$*$ topology i.e. when integrating the spectral measure with respect to bounded continuous functions instead of polynomials, and in this situation the finite moments condition can be reduced to the existence of a finite second moment. 
In the multivariate setting this was 
generalized to the result that independent Wigner matrices and 
deterministic matrices are asymptotically free, in the sense of Voiculescu's free probability theory. More precisely, if $\mbf X$ 
denotes a collection of independent Wigner matrices and $\mbf A$ a collection of deterministic 
matrices, then we have for any $^*$-polynomial $p$
		$$\E\Big[ \frac 1 N \Tr \, p (\mbf X,\mbf A) \Big] \limN \Phi \big[p(\mbf x, \mbf a)\big],$$
where $\mbf x$ denotes a free semicircular system,
$\mbf a$ is distributed according to the limiting $^*$-distribution of $\mbf A$,
and $\mbf x$ and $\mbf a$ are free. The case of {\sc gue} matrices and general deterministic matrices goes back to Voiculescu \cite{Voi91,Voi98}, Dykema \cite{Dyk93} considered general Wigner matrices, but special block-diagonal deterministic matrices; the proof of the general case can be found in \cite{AGZ,MSp}.

The study of fluctuations of linear statistics around their expectation 
was initiated by Jonsson \cite{Jon82} for the slightly different model of Wishart matrices, and 
then by Khorunzhy, Khoruzhenko and Pastur \cite{KKP96} for a Wigner matrix. Since then, many more results were produced on the fluctuation of linear spectral 
statistics of a single random matrix \cite{SS98,Joh98, CLT1,CLT2, CLT3, CLT4}. Our concern in this paper is about the fluctuations 
of linear statistics in several matrices. This direction of research was initiated 
by two of the present authors \cite{MS06}, with the study of centered traces of 
$^*$-polynomials in independent complex Gaussian and Wishart matrices, 
and has been developed further in \cite{Red11, JL20, DF19}. It turns out that in all these situations, 
one can prove a central limit theorem (CLT) for the centered traces, namely
	\eq
		   \Tr \, p(\mbf X, \mbf A) - \E\big[   \Tr \, p(\mbf X, \mbf A) \big]
	\qe
converges to a Gaussian random variable, and the convergence holds jointly 
for all $^*$-polynomials $p$. The absence of normalization is a 
first remarkable common fact in all these CLT, which tells that the eigenvalues of 
random matrices fluctuate much less than i.i.d. random variables. 

Another aspect that appears 
in the fluctuation problem, is a breaking of universality compared to the first 
order problem:
\begin{itemize}
	\item the limiting fluctuations of linear spectral statistics of a Wigner matrix 
	depend on the pseudo-variance $\E(x_{ij}^2)$, on the diagonal variance $\E(x_{ii}^2)$, and on the fourth moment 
	$\E[ x_{ij}^2 \bar x_{ij}^2]$ of its non diagonal entries.
	\item the asymptotic second order freeness theory of complex random 
	matrices and of real random matrices are different. 
\end{itemize}

In this article, we extend the result in \cite{MS06} by considering the fluctuations of traces of several independent Wigner and deterministic matrices, proving a CLT under mild assumptions. 
It turns out that Wigner matrices and deterministic matrices are, in contrast to  {\sc gue}  and deterministic matrices, \emph{not} free of second order in the sense of
\cite{MS06}. But still there is a very definite structure governing the asymptotic behaviour; however, a crucial observation is that the limiting fluctuations do not depend only on the limiting $^*$-distribution of the deterministic matrices but on more information. This puts this setting very canonically into the frame of {traffic probability theory}, which was developed by one of present authors \cite{Mal20}. One can see the present investigations as a first step for a more general treatment of global fluctuation by \emph{the traffic approach}.

The paper is organized as follows. In Section \ref{Sec:presentation}, we will present our main results.
In the general case, the transpose of the deterministic matrices will also show up in the formula for the asymptotic covariance, thus one needs for the most general version of our results also assumptions on the joint asymptotic distribution of the the deterministic matrices and their transposes. In the special case where the  Wigner matrices have vanishing pseudo-variance the transpose does not play a role; hence we will also have a separate discussion for this case.
In Sections \ref{Sec:Tightness} and \ref{Sec:CharLim} we will provide the proofs of our main theorems.

\section{Presentation of the results}\label{Sec:presentation}

\subsection*{Assumptions on Wigner and deterministic matrices}

Firstly, we list the notations and assumptions on the matrices under consideration.

\begin{hypothesis}\label{hyp:X} Let  $X=\frac 1{\sqrt N}({x_{ij}})_{i,j}$ be a Wigner matrix.
	\begin{enumerate}
		\item We assume $\E(|x_{12}|^2) = 1$. 
		\item We assume $X$ is invariant in law by conjugation by permutation matrices, or equivalently $x_{12} \overset{\mcal Law}= \bar x_{12}$.
	\end{enumerate}

\end{hypothesis}

The first condition is a normalization condition, while the second one is 
technical and inherent to our proof. We hope that eventually we can get rid of this last one by 
improving the first steps of our method.

\begin{definition}\label{def1}
The triple $( \E(x_{12}^2), \E(x_{11}^2), k_4)$	
is called the \emph{parameters} of a Wigner matrix $X=\frac 1 {\sqrt N} ({x_{ij}})_{i,j}$,  where $k_4$ is the following fourth cumulant of the off-diagonal entry
		$$k_4 = k_4(x_{12},x_{12},\bar x_{12},\bar x_{12})=\E[ x_{12}^2 \bar x_{12}^2] - 2 - |\E[x_{12}^2]|^2.$$ 
We call $\E(x_{12}^2)$ the \emph{pseudo-variance} and $\E(x_{11}^2)$ the \emph{diagonal variance} of the Wigner matrix.
\end{definition}

Note that a {\sc gue} 
matrix has parameters $(0,1,0)$ and a {\sc goe} matrix has parameters $(1,2,0)$. A real Wigner 
matrix has always pseudo-variance equal to one. 

Now we introduce the hypotheses on deterministic 
matrices. These assumptions imply a CLT in the simpler case where the Wigner matrices admit a vanishing pseudo-variance. The general case requires some more assumptions. For two matrices $A$ and $B$, we let $A\circ B$ be the
entry-wise product of the matrices, also called Hadamard or
Schur product.

\begin{hypothesis}\label{hyp:A} The collection $\mbf A=(A_j)_{j\in J}$ of deterministic matrices is assumed to satisfy
	\begin{enumerate}
		\item $\sup_{N} \| A_{j}\| < \infty$ for any $j\in J$, for $\| \, \cdot\, \|$ the operator norm.
		\item For any $^*$-polynomials $p$ and $q$, 
			$$\phi_{(\circ)}(p, q):= \phi(p \circ q):= \underset{ N \rightarrow \infty}\lim \frac 1 N \Tr \big[ p(\mbf A) \circ q(\mbf A) \big]$$
		exists as $N$ goes to infinity.
	\end{enumerate}
\end{hypothesis}

\begin{definition}\label{not:parameterA}
Under Hypothesis \ref{hyp:A},
we call $\phi_{(\circ)}$ the \emph{parameter} of $\mbf A$. 
\end{definition}
From Hypothesis \ref{hyp:A}, we have the existence of the limiting $*$-distri\-bution of $\mbf A$: $\phi(p) = \phi(p \circ \mbb I)=\phi(\mbb I \circ p)$,
where $\mbb I$ is the unit *-polynomial.
Also the first assumption in Hypothesis \ref{hyp:A} implies that, up to a subsequence, the second assumption 
is always satisfied, since $\big| \frac 1 N \Tr  [ A \circ B  ] \big| \leq \|A\| \times \|B\|$ 
for any matrices $A,B$. The limit $\phi_{(\circ)}$ will be used to describe 
the covariance of the limiting Gaussian process. 

\subsection*{Statements on convergence}
Omitting momentarily the description 
of this  covariance, our main result can be stated as follows.

Let us first treat the special case where the Wigner matrices have vanishing pseudo-variance.

\begin{theorem}\label{MainTh1} 
Let $\mbf X$ be a collection of independent Wigner matrices, such that each Wigner matrix satisfies Hypothesis \ref{hyp:X}, and let $\mbf A$ be a collection of  deterministic matrices satisfying Hypothesis \ref{hyp:A}. In addition assume that 
all the Wigner matrices of $\mbf X$  have vanishing pseudo-variance. For any $^*$-polynomial $p$, 
we denote 
	$$Z_N(p) = \Tr \, p(\mbf X, \mbf A) - \E\big[   \Tr \, p(\mbf X, \mbf A) \big].$$
Then the process $\big(Z_N(p)\big)_{p}$ converges to a Gaussian process 
$\big(z(p)\big)_{p}$. The second order $^*$-distribution
	$$\phi^{(2)}: (p,q) \mapsto   \E\big[ z(p) z(q)\big]$$
depends only on the parameters of the matrices as given in Definitions \ref{def1} and \ref{not:parameterA}.
\end{theorem}

Note that the covariance of the process is completely determined 
by the second-order distribution via the formula 
$\E\big[z(p)\overline{z(q)}\big] = \phi^{(2)}(p,q^*)$. 

We now turn to the case of general Wigner matrices, without any assumption on the pseudo-variance. In this case we also have to control the limit of expressions involving the transpose of the deterministic matrices.
We denote by $A^t$ the transpose of a matrix $A$. 

\begin{theorem}\label{MainTh3}
Let $\mbf X$ be a collection of independent Wigner matrices, such that each Wigner matrix satisfies Hypothesis \ref{hyp:X}, and let $\mbf A$ be a collection of  deterministic matrices satisfying Hypothesis \ref{hyp:A}. In addition assume that for any $^*$-polynomials $p$ and $q$, the following limit exists
			$$\phi_{(t)}(p, q) := \phi(p\, q^t):= \underset{ N \rightarrow \infty}\lim \frac 1 N \Tr \big[ p(\mbf A) q(\mbf A)^t \big].$$
Then the conclusion of Theorem \ref{MainTh1} is valid, with a limiting second-order $^*$-distribution $\phi^{(2)}(p,q)$ that also depends on $\phi_{(t)}$.
\end{theorem}

\subsection*{Description of the second-order distribution}
First, in Theorem \ref{MainTh2}, we will describe  the second-order distribution in the case of vanishing pseudo-variance and unit diagonal variance for all the Wigner matrices (i.e., the parameters for each Wigner matrix are $(0,1,k_4)$). In Theorem \ref{MainTh4} we will then extend this to the general case.

For some $m,n\geq 1$, we consider Wigner matrices $X_1 \etc X_{m+n}$ of $\mbf X$, with possible repetitions of the matrices. Moreover, without loss of
generality we assume that the family $\mbf A$ of deterministic matrices
is stable by
$^*$-monomials (that is $p(\mbf A)$ is an element of $\mbf A$ for any $^*$-monomial $p$). We consider deterministic matrices $A_1\etc A_{m+n}$ of $\mbf A$. Defining the random matrices 
	\eq
		p_N = X_1A_1\cdots X_mA_m, & & q_N = X_{m+1}A_{m+1} \cdots X_{m+n}A_{m+n},
	\qe
and the associated polynomials
	\eq
		p =x_1a_1\cdots x_ma_m, & & q = x_{m+1}a_{m+1} \cdots x_{m+n}a_{m+n},
	\qe
in order to completely characterize $\phi^{(2)}$ it is sufficient to give a formula for $\phi^{(2)}(p,q)$, the limit of $$\E\big[ \big(\Tr\, p_N - \E[\Tr \, p_N]\big) \big(\Tr \, q_N - \E[\Tr \, q_N]\big)\big].$$ 

\setbox1=\hbox{\includegraphics{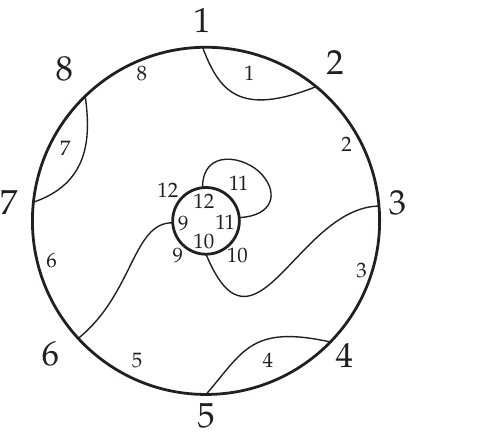}}

\setbox2=\vbox{\hsize 200pt\raggedright
Figure 1. The non-crossing annular pairing $\sigma =
(1,2)(3,10)\ab(4,5)(6,9)\ab(7,8)\ab(11,12)$. Its Kreweras
complement is $K(\sigma)= (1)\ab(\ab 6,\ab
8, 2,\ab 10,\ab 12)\ab(3, 5, 9)(4)(7)(11)$. We have $\phi_{K(\pi)}(a_1, \dots, a_{12}) = \phi(a_1) \ab
\phi(a_7) \ab \phi(a_4)\ab\phi( a_6a_8a_2 
a_{10}a_{12})\ab \phi(a_3 a _5   a_9) \phi( a_{11})$ and 
$\widetilde\phi_{K(\pi)}(a_1, \dots, a_{12}) = \phi(a_1) \ab
\phi(a_7) \ab \phi(a_4)\ab\phi( a_6a_8a_2 \circ
a_{10}a_{12})\ab \phi(a_3 a _5 \circ a_9) \phi( a_{11})$.}

\begin{figure}[t]
$\vcenter{\box1}\vcenter{\box2}$
\end{figure}

\medskip
To give such a formula, we need to recall a few combinatorial facts about annular versions of non-crossing partitions. 
We refer to \cite{MN04} for the background definitions on annular
non-crossing permutations. We denote by $NC_2(m,n)$ the non-crossing pairings on an
$(m,n)$-annulus with at least one through string 
 and by $NC_2^{(l)}(m,n)$ the non-crossing
pairings on an $(m,n)$ annulus with exactly $l$ through
strings. We let $\gamma_{m,n}$ be the permutation $(1, 2, 3,
\dots, m)(m+1, \dots, m+n)$ and we put, for a pairing $\sigma\in NC_2(m,n)$, $K(\sigma)
= \sigma \gamma_{m,n}$; this is a non-crossing permutation called the \textit{Kreweras
complement} of $\sigma$. We define
	\eqa\label{Eq:DefPhiK}
		\phi_{K(\sigma)}(a_1,\ab \dots,\ab a_{m+n}) = \prod_{(i_1\etc i_\ell) \in K(\sigma)}  \phi(a_{i_1} \cdots a_{i_\ell}),
	\qea
where the notation means that the product is over the cycles $(i_1\etc \ab i_\ell)$ of the permutation $K(\sigma)$ (note that, since $K(\sigma)$ is non-crossing, the cycles respect the order on each of the two circles) (see Figure 1) and we recall that $\phi$ is the limiting $^*$-distribution of the deterministic matrices.

We shall need a modification $\widetilde \phi_{K(\sigma)}$ of $\phi_{K(\sigma)}$.  
Let $\sigma \in NC_2^{(2)}(m,n)$, then $K(\sigma)$ will have
exactly two through cycles. Let us write these through
cycles as $(i_1, \dots, i_k, i_{k+1}, \dots, i_l)$ and $(j_1, \dots , j_{k'}, j_{k'+1}, \dots, j_{l'})$, with
$i_1, \dots, i_k\in [m]$ and $i_{k+1},
\dots, i_l\in [m+1, m+n]$, and $j_1, \dots, j_{k'}\in [m]$ and $j_{k'+1},
\dots, j_{l'}\in [m+1, m+n]$. We define then
$\widetilde\phi_{K(\sigma)}(a_1, \dots,\ab a_{m+n})$ by making 
the following two replacements in \eqref{Eq:DefPhiK}:
	\eq
	\phi(a_{i_1} \cdots a_{i_k} a_{i_{k+1}} \cdots a_{i_l})
   & \mapsto &  \phi_{(\circ)}(  a_{i_1} \cdots a_{i_k}, 
  a_{i_{k+1}} \cdots a_{i_l}),\\
 	\phi(a_{j_1} \cdots a_{j_{k'}} a_{j_{k'+1}} \cdots
  a_{j_{l'}}) & \mapsto & \phi_{(\circ)}( a_{j_1} \cdots
  a_{j_{k'}},a_{j_{k'+1}} \cdots a_{j_{l'}}),
  	\qe
where $\phi_{(\circ)}$ is the bilinear form of Hypothesis \ref{hyp:A}. This is illustrated in Figure 1.

\begin{theorem} \label{MainTh2}
Under the assumption
of vanishing pseudo-variance and unit diagonal variance for all involved Wigner matrices,
the second-order distribution in Theorem \ref{MainTh1} is given by
	\begin{multline}\label{eq:thm2}
    \phi^{(2)}(p,q) =\\
      \kern-1em
    \mathop{\sum_{\sigma \in NC_2(m,n)}}_{\mrm{non-mixing}}
      \kern-1em
    \phi_{K(\sigma)}(a_1, \dots, a_{m+n})
    +  \kern-1em
\mathop{\sum_{\sigma \in NC_2^{(2)}(m,n)}}_{\mrm{non-mixing}}
  \kern-1em
k_{4,\sigma}\widetilde\phi_{K(\sigma)}(a_1, \dots, a_{m+n}).
\end{multline}
The condition that $\sigma$ 
is \emph{non-mixing} means that labels 
associated to different
Wigner matrices belong to different cycles of $\sigma$. In the second sum we also require in addition that the four Wigner matrices involved in the two through cycles must all be the same.
The value $k_{4,\sigma}$ is then the parameter of the Wigner matrix
corresponding to the two through cycles.

\end{theorem}

\begin{figure}[t]\setcounter{figure}{1}
\leavevmode\kern-3em
\hbox{\includegraphics{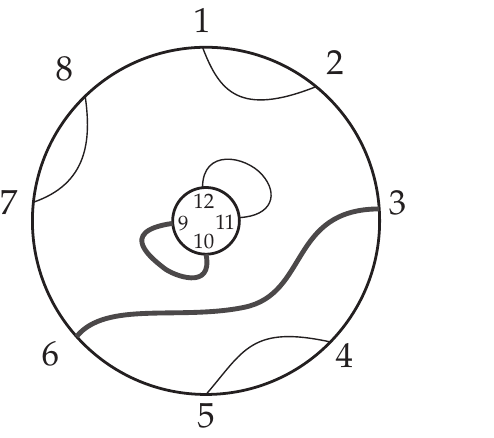}\includegraphics{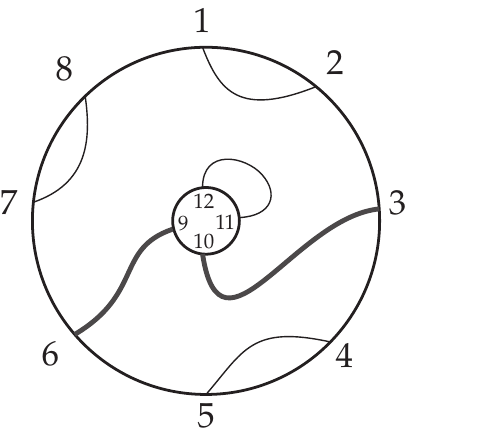}\includegraphics{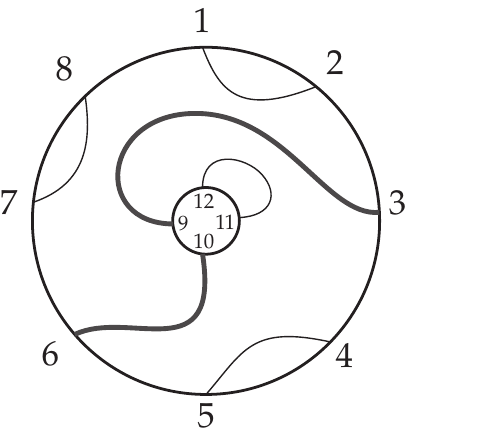}}

\caption{The non-crossing pairing $\sigma =
  (1,2)(3,6)\ab(4,5)\ab(7,8)\ab(9,10)\ab(11,12) \in
  NC_2(8)\ab \times NC_2(4)$. Modifying $(3,6)$ and $(9,10)$, we can produce two
  elements of $NC_2^{(2)}(8,4)$ on the right depending on
  how we connect the through strings.}

\end{figure}

Let us now extend our description of the covariance to the general case without any assumption on the the parameters of the Wigner matrices.
The second-order distribution is then a slight modification of the function \eqref{eq:thm2} of Theorem \ref{MainTh2} that takes into account the pseudo-variance and the diagonal variance parameters. With the same notations as before, consider $\sigma \in NC_2^{(\ell)}(m,n)$ with $\ell$ through strings and let $\theta_\sigma = \theta_{k_1}\cdots \theta_{k_\ell}$ be the product of the pseudo-variance parameters of the Wigner matrices $X_{k_1} \etc X_{k_\ell}$ associated to the through strings. Without loss of generality, we assume that $\mbf A$ is closed under the transpose, i.e., if $A\in\mbf A$, then also $A^t\in\mbf A$. For a monomial $m=x_1a_1 \cdots x_na_n$ we denote $s(m) = x_na_{n-1}^t x_{n-1} \cdots a_1^tx_1 a_n^t$ and extend this definition of $s$ by linearity. We extend also the definition of $\phi$ by setting $\phi(a_i  a_j^t ) = \phi_{(t)}(a_i,a_j)$, of $\phi_{(\circ)}$ by $\phi_{(\circ)}(a_i,a_j^t)=\phi_{(\circ)}(a_i,a_j)$.

Let $\sigma \in NC_2^{(1)}(m,n)$, then $K(\sigma)$ will have
exactly one through cycle. As before we only have to consider non-mixing $\sigma$, thus the two involved Wigner matrices in the through cycle of $\sigma$ will be the same. We denote by $\eta_\sigma$ and $\theta_\sigma$ the diagonal variance and the pseudo-variance of this Wigner matrix. We write the through
cycle of $K(\sigma)$ as $(i_1, \dots, i_k, i_{k+1}, \dots, i_l)$ with
$i_1, \dots, i_k\in [m]$ and $i_{k+1},
\dots, i_l\in [m+1, m+n]$ and define 
$\widetilde\phi_{K(\sigma)}(a_1, \dots, a_{m+n})$ by making 
the following replacement in the factor of
$\phi_{K(\sigma)}(a_1,\ab \dots,\ab a_{m+n})$ in \eqref{Eq:DefPhiK}:
	$$\phi(a_{i_1} \cdots a_{i_k} a_{i_{k+1}} \cdots a_{i_l}) 
   \mapsto \phi_{(\circ)}(  a_{i_1} \cdots a_{i_k}, 
  a_{i_{k+1}} \cdots a_{i_l}).$$

\begin{theorem}\label{MainTh4} 
For arbitrary parameters $(\theta,\eta,k_4)$ for each of the involved Wigner matrices, the second-order distribution in Theorem \ref{MainTh3} is
\begin{align}\label{eq:formula-phi2}
	\notag	\phi^{(2)}(p,q) &=
 \kern-1em
    \mathop{\sum_{\sigma \in NC_2(m,n)}}_{\mrm{non-mixing}}
      \kern-1em
   [ \phi_{K(\sigma)}(a_1, \dots, a_{m+n})+\theta_\sigma \cdot (\phi_{(t)})_{K(\sigma)}(a_1, \dots, a_{m+n})]\\
&\quad
+  \kern-1em
\mathop{\sum_{\sigma \in NC_2^{(2)}(m,n)}}_{\mrm{non-mixing}}
  \kern-1em
k_{4,\sigma}\cdot \widetilde\phi_{K(\sigma)}(a_1, \dots, a_{m+n}).\\
\notag
		&\quad+   \kern-1em \mathop{\sum_{ \sigma \in NC_2^{(1)}(m,n)}}_{\mrm{non-mixing}}\big(\eta_\sigma - 1 - \theta_\sigma\big) \cdot \tilde \phi_{K(\sigma)}(a_1 \etc a_{m+n}).
\end{align}
\end{theorem}

Note that the third and the fourth term in Equation \eqref{eq:formula-phi2} cannot appear together. The set $NC_2(m,n)$ is only non-empty if $m+n$ is even; if both $m$ and $n$ are even then the number of through cycles of a $\sigma\in NC_2(m,n)$ must be even and then $NC_2^{(1)}(m,n)$ is empty; if both $m$ and $n$ are odd then the number of through cycles of $\sigma$ must be odd and then $NC_2^{(2)}(m,n)$ is empty.

Let us consider some special cases of our Theorem \ref{MainTh4}.
\begin{enumerate}
\item
If the pseudo-variance is zero and the diagonal variance is equal to one, then the second and the fourth summand in Equation \eqref{eq:formula-phi2} vanish and the result reduces to Equation \eqref{eq:thm2} of Theorem \ref{MainTh2}.
\item
For  {\sc gue} random matrices, with parameters $(0,1,0)$, only the first summand in Equation \eqref{eq:formula-phi2} is different from zero and reproduces the result for second order freeness between {\sc gue} and deterministic matrices from \cite{MS06}.
\item 
For  {\sc goe} random matrices with parameters $(1,2,0)$, the combination $\eta - 1 - \theta$ is again zero and we have only the first two terms in Equation \eqref{eq:formula-phi2}; this yields the formula of Redelmeier \cite{Red11} for real second order freeness between {\sc goe} and deterministic matrices.
\item
For  complex or real Wigner matrices for which the parameters $(\eta,\theta)$ agree with the corresponding Gaussian ensembles, the last term in 
 Equation \eqref{eq:formula-phi2} vanishes and we get in addition to the {\sc gue} and {\sc goe} case the term involving $k_4$. For the case of trivial deterministic matrices, such results go back to Khorunzhy, Khoruzhenko, and Pastur \cite{KKP96}.
\end{enumerate}

\begin{example} 
\begin{enumerate}
	\item We have by a direct computation
	\eq
		\phi^{(2)}(xa_1, xa_2) & = & \Nlim \E\big[ \sum_{i,j,i',j'} A_1(i,j)A_2(i',j') X(j,i)X(j',i')\big]\\
		& = &  \Nlim  \frac 1 N\Big( \sum_{i\neq j} A_1(i,j)A_2(j,i) +  \sum_{i\neq j} A_1(i,j)A_2(i,j)\theta \\
		& & + \sum_i A_1(i,i)A_2(i,i)\eta\Big)\\
		& = & \Nlim  \frac 1 N\Big( \sum_{i,j} A_1(i,j)A_2(j,i) +  \sum_{i,j} A_1(i,j)A_2(i,j)\theta \\
		& & + (\eta - 1 - \theta) \sum_i A_1(i,i)A_2(i,i)\Big)\\
		& =& \phi(a_1a_2) + \phi(a_1a_2^t) +(\eta - 1 - \theta) \phi_{(\circ)}(a_1, a_2),
	\qe
where $\eta$ and $\theta$ are the diagonal variance and pseudo-variance of the Wigner matrix $X$. The three terms indeed correspond to three of the four terms in Theorem \ref{MainTh4}, since for each term there is a single $\sigma \in NC(1,1)$ which consists in the permutation with a single cycle. Note that the  term involving $k_4$ does not play a role for this fluctuation of second order, as there is no contributing $\sigma$ with two through cycles.
	\item The formula gives for the fluctuation of moments of fourth order
	\eq
		\lefteqn{\phi^{(2)}(x_{\ell_1}a_1x_{\ell_2}a_2,x_{\ell_3}a_3x_{\ell_4}a_4)}\\
		& = & \Big( \delta_{\ell_1,\ell_3} \delta_{\ell_2,\ell_4} \phi(a_1 a_4) \phi(a_2a_3) + \delta_{\ell_1,\ell_4}\delta_{\ell_2,\ell_3} \phi(a_1 a_3) \phi(a_2 a_4)\Big)\\
		&  & + \kappa_{4,\ell_1}  \Big( \delta_{\ell_1,\ell_2,\ell_3,\ell_4}  \phi_{(\circ)}(a_1,a_4) \phi_{(\circ)}(a_2,a_3) +  \phi_{(\circ)}(a_1,a_3) \phi_{(\circ)}(a_2,a_4) \Big)\\
	& &  + \theta_{\ell_1}\theta_{\ell_2}\Big( \delta_{\ell_1,\ell_3} \delta_{\ell_2,\ell_4}  \phi(a_1 a_3^t) \phi(a_2a_4^t) + \delta_{\ell_1,\ell_4}\delta_{\ell_2,\ell_3} \phi(a_1 a_4^t) \phi(a_2 a_3^t)\Big).
	\qe
\end{enumerate}
\end{example}

The proofs of Theorems \ref{MainTh1}-\ref{MainTh4} are given in the rest of the paper.

\section{Boundedness of moments}\label{Sec:Tightness}

In the present and the following sections, we consider random
matrices $M_1 \etc M_n$ of the following form: for positive
integers $p_1\etc p_n$, with $m_j=p_1+ \cdots + p_{j-1}$, $m = m_{n+1}$,
\eqa\label{Def:Mj}
	\begin{array}{cc}
	M_j = X_{m_j+1}A_{m_j+1} \cdots X_{m_j+p_j}A_{m_j+p_j}, & \forall j = 1\etc n,\end{array}
\qea
where $X_1\etc X_{m}$ are Wigner matrices in $\mbf X$ and 
$A_1\etc A_{m}$ are deterministic matrices in $\mbf A$.
As in the previous section, we allow
repetitions of matrices. We denote the random variables and the complex numbers
\eqa\label{Def:Z}
	Z_j & = & \Tr ( M_j) - \E \big[  \Tr (  M_j ) \big], \ \forall
\, i=j\etc n,\\
	\tau^{(2)}& = & \E[Z_1 \cdots Z_n], \ \ \ \tau^{(1)} = \E\Big[ \prod_{j=1}^n \frac 1 N  \Tr M_j \Big] .
\qea

The purpose of this section is to prove that $\tau^{(2)}$ is bounded as $N$ goes to infinity and to give a first combinatorial description of the leading order of $\tau^{(2)}$.  

\subsection{General scheme for the study of the statistics $\tau^{(1)}$ and $\tau^{(2)}$}
Here we summarize the general ideas for the study of the asymptotics of $\tau^{(2)}$ (as well as $\tau^{(1)}$). 
The basic argument is to translate the computations of moments in terms of a series of graphs. 
\subsubsection{Preliminary encoding: Labeled graph, quotient graphs and subgraphs}
One can check that
\eqa
		\tau^{(1)} &= & \E\Big[ \prod_{j=1}^n \frac 1 N  \Tr [X_{m_j+1}A_{m_j+1}\cdots X_{m_j+p_j}A_{m_j+p_j}]\Big].
\qea
This is encoded in terms of a \emph{labeled graph} $T$ 
	   (Definition \ref{Def:LabGraph}). It consists of a disjoint union $T=T_1\sqcup \cdots \sqcup T_n$ of $n$ simple directed cycles, where $T_j$ has $2p_j$ edges with alternating labels
	      $x_{m_i+1},a_{m_i+1},\dots, x_{m_i+p_1},a_{m_i+p_i}$ in the opposite sense of the direction of the  cycle, see (a) in Figure \ref{Fig:Table}. Labeled graphs are special cases of the test graphs defined in \cite{Mal20}.

Now one can write

\eqa
                    \tau^{(1)} =&N^{-n} \sum_{ \mbf i} \beta^{(1)}_X(\mbf i) \times \beta_A (\mbf i),
                \label{Eq:Not:1}
\qea
where the sum is over all collections of multi-indices $\mbf i = \big(  i_{(j,k)} , i_{(j,k)'} | j\in [n], k\in [p_j] \big)$  in $[N]$. Here  we have set
	\eqa
	       \quad\beta^{(1)}_X\big(\mbf i \big)  &= & \prod_{j=1}^n\E\big[ X_{m_j+1}(i_{(j,1)},i_{(j,1)'})
                  \cdots X_{m_j+p_j}(i_{(j,p_j)},i_{(j,p_j)'}) \big],	\label{Eq:Not:1.1}\\
                \quad\beta_A\big( \mbf i \big) & = &   \prod_{j=1}^nA_{m_j+1}(i_{(j,1)'},i_{(j,2)}) \cdots A_{m_j+p_j}(i_{(j,p_j)'},i_{(j,1)}).
         \qea
         
 For any multi-index $\mbf i$ as above, we denote by $\pi=\ker(\mbf i)$ the partition of the set $V=\{ (j,k), (j,k)'   | j \in [n], k\in [p_j]\}$ 
defined as follows: for $v,w\in V$ we have  $v\sim_{ \pi} w$ if and only if  $i_v=i_w$. 
To such a partition $\pi$, one can associate \emph{the quotient graph $T^\pi$}: the labeled graph  $T^\pi$ is obtained by identifying vertices of $T$ that 
	    belong to the same block of $\pi$ (see the general definition in Definition \ref{Def:InjTrace}).

Now, the invariance by permutation of the Wigner matrices implies that each quantity $\beta^{(1)}_X(\mbf i)$ depends only on $\pi= \ker(\mbf i)$. We denote it by $\beta^{(1)}_X(\pi)$. We then can write
	\begin{equation}\label{Eq:Not:2}
		 \tau^{(1)}  =   \sum_{\pi \in \mcal P(V)} N^{-n} \beta^{(1)}_X(\pi)  \times \beta_A(\pi)=: \sum_{\pi \in \mcal P(V)} \alpha^{(1)}(\pi),
	\end{equation}
where the sum is over all partitions of $V$ and
	\begin{equation}\label{Eq:Not:3}
\beta_A(\pi)  = 
	\sum_{\substack{ \mbf i\in [N]^{V} \\ \ker(\mbf i) = \pi}}   \prod_{j=1}^nA_{m_j+1}(i_{(j,1)'},i_{(j,2)}) \cdots A_{m_j+p_j}(i_{(j,p_j)'},i_{(j,1)}).
	\end{equation}
The $\beta$-functions are expressed in terms of subgraphs of $T^\pi$. The contribution $\beta_A$ of the deterministic matrices is given by \eqref{Eq:Not:3}. Its expression (Definition \ref{Def:TEC}) is written in terms of the subgraph $T_A^\pi$ of $T^\pi$, which has the same vertices as 
$T^\pi$, and has only those edges of $T^\pi$ that are labeled by $a_1\etc a_{m}$ (i.e, edges labeled by $x_k$ do not appear in $T_A^\pi$) , see the leftmost graph of (d) in Figure \ref{Fig:Table}. One can similarly define $T_X^\pi$ (see the rightmost graph (d) in Figure \ref{Fig:Table}).
The benefit is that in \eqref{Eq:Not:2}, the sum over $\pi$ is finite (independent of $N$) and we have separated the contributions of the Wigner and the deterministic matrices.
	
This preliminary encoding is illustrated in Figure \ref{Fig:Table} below.

\begin{figure}[!t]
\centering 
\subfigure[Solid lines: the labeled graph $T=T_1\sqcup T_2\sqcup T_3\sqcup T_4$. Dashed lines: a partition with 6 blocks of the vertex set $\pi = \big\{\{ 1,4\} , \{1',4'\}, \{2,3'\}, \{2',3\}, \ab \{5, 6', \ab7, 8'\}, \{5', 6, 7', 8\} \big\}$.]{\includegraphics{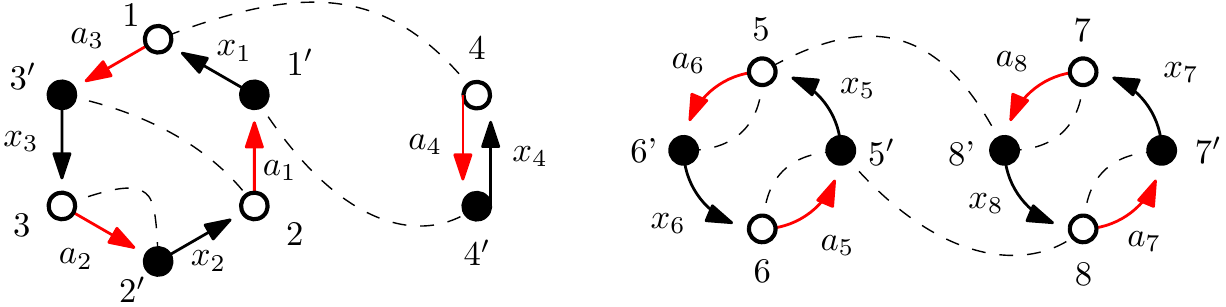}}\\
\subfigure[The quotient graph $T^\pi$, obtained by identifying the  vertices connected by dashed  lines.]{\includegraphics{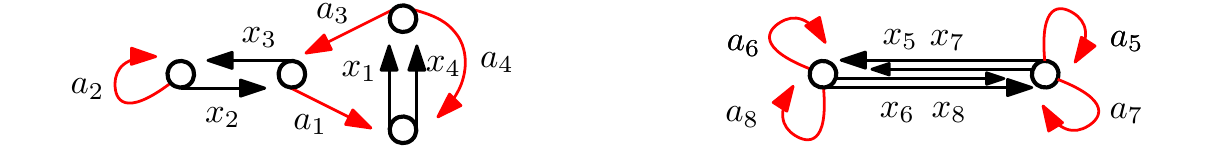}}\\\subfigure[The subgraphs $T^\pi_1\etc T^\pi_4$, obtained by identifying, on each of the subgraphs $T_1$, $T_2$, $T_3$, and $T_4$, the  vertices connected by dashed  lines.]{\includegraphics{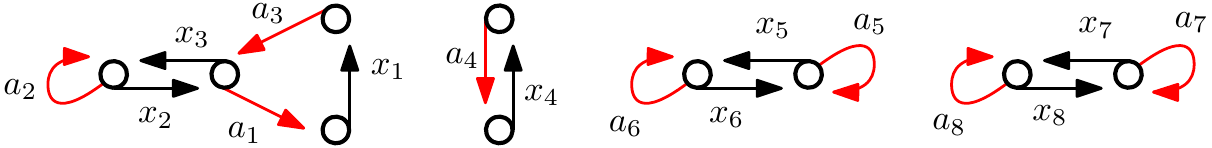}}\\
\subfigure[Left: the subgraph $T^\pi_A$. Right: the subgraph $T_X^\pi$.]{\includegraphics{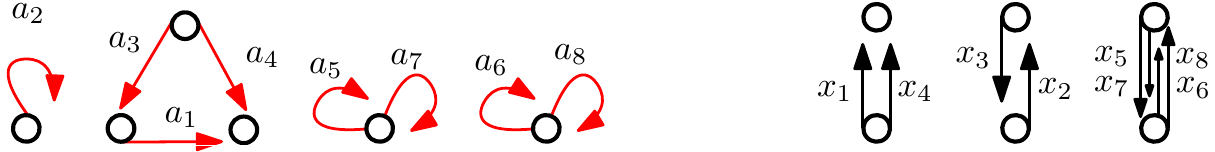}}\\
\caption{Main graphs used for the proof of main theorems.}
  \label{Fig:Table}
\end{figure}

The same method can be used to investigate $\tau^{(2)}$. Similarly to Formula \eqref{Eq:Not:2}, one has that
	 \eqa\label{Eq:UnformalInjective}
	 	\tau^{(2)} = \sum_{\pi\in \mcal P(V)} \alpha^{(2)}(\pi).
	\qea
Each summand 
$\alpha^{(2)}(\pi)$ is parametrized by the quotient graph $T^\pi$, see (b) in Figure \ref{Fig:Table}.
Because of the centeredness of the variables appearing in the second-order statistic, the expression of $\beta^{(2)}_X(\pi)$ is more convoluted. It is determined by the quotient graphs $T_{j,X}^\pi$, $j\in J\subset [1, n]$, for some subset $J$.

 \subsubsection{Boundedness of the statistics}
 
 We can then actually prove the boundedness of 
	\eqa\label{Eq:Note:4}
		\alpha^{(1)}(\pi) = N^{-n} \beta^{(1)}_X(\pi)  \times \beta_A(\pi)
	\qea
 for any $\pi$ as $N$ goes to infinity. We mention briefly the next steps for the convergence of $\tau^{(1)}$ (this is valid for any permutation invariant matrices, not just for Wigner matrices).
\begin{enumerate}
	\item [i)]Find appropriate normalizations $\omega^{(1)}_X(\pi) = N^{-n_X(\pi)}\beta_X(\pi)$ and $\omega_A(\pi) = N^{-n_A(\pi)}\beta_A(\pi)$ such that $\omega^{(1)}_X$ and $\omega_A$ are bounded functions. 
	\item [ii)] Prove that whenever $q(\pi):=n_X(\pi)+n_A(\pi)>n$ then $\omega^{(1)}_X(\pi)=0$. 
	\item [iii)] Identify the valid partitions $\pi$, i.e., those for which $q(\pi)=n$.
\end{enumerate}
We use the same approach for $\tau^{(2)}$.
One can show that the correct normalization of the contribution of the Wigner matrices  (Definition \ref{Def:OmegaX}) is $\omega^{(2)}_X(\pi) = N^{\frac{m}2}\beta^{(2)}_X(\pi)$ (as for  $\omega^{(1)}_X(\pi)$), where $\beta^{(2)}_X(\pi)$ is defined after Eq.~(\ref{eq:QuotientPart:1}).
Identifying the appropriate normalization of the contribution of the deterministic matrices  $\beta_A(\pi)$ is more involved. The sharp bounds of \cite{MS11} imply that the optimal normalization is $\omega_A(\pi) = N^{- \frac{\mathfrak f(\pi)}2}\beta_A(\pi)$, where $\mathfrak f(\pi)$ is the number of leaves of the \emph{forest of two-edge connected components} (t.e.c.c.) $\mcal T\hspace{-2pt}ec(T^\pi_A)$  of $T_A^\pi$ (Definition \ref{Def:TEC}). The constant $\mathfrak f(\pi)$ may take any integer value greater than or equal to the number of connected components $c(\pi)$ of $T_A^\pi$. See the two first graphs in Figure \ref{Fig:TableAux}.

\begin{figure}[!t]
\centering

{\includegraphics{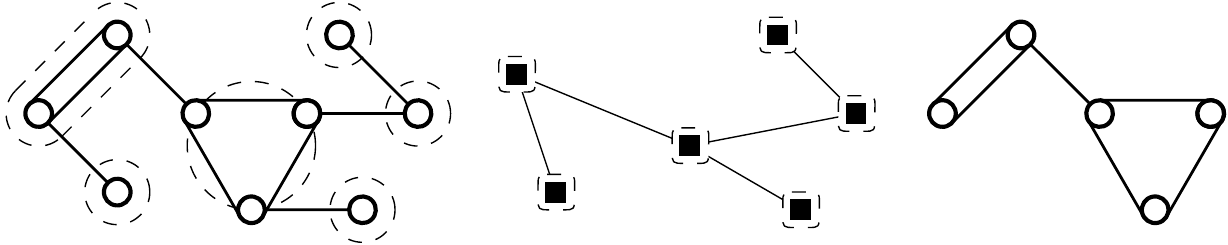}}
\caption{Auxiliary graphs for Section \ref{Sec:Tightness}. From left to right: a graph $T$ (dashed lines encircle the two-edge connected components of $T$); the tree of two-edge connected components of $\mcal T\hspace{-2pt}ec(T)$; the pruning of $\mcal Prun(T)$.}
  \label{Fig:TableAux}
\end{figure}

From the previous steps, one has that
	 \eq
	 	\alpha^{(1)}(\pi)  = & N^{q(\pi)-n}\omega^{(1)}_X(\pi) \omega_A(\pi), \quad
	 	\alpha^{(2)}(\pi)  = & N^{q(\pi)}\omega^{(2)}_X(\pi) \omega_A(\pi), 
	\qe
 where the $ \omega$-functions are bounded and 
\begin{equation}\label{def:q}
q(\pi) = -\frac{m}2  +\frac{ \mathfrak f(\pi)}2.
\end{equation}  
Consider a partition $\pi \in \mcal P(V)$, and hence a subgraph $T^\pi$, such that the number of leaves $\mathfrak f(\pi)$ in the forest of t.e.c.c. of $T^\pi_A$ equals twice the number of connected components $c(\pi)$ of $T^\pi_A$, so that $q(\pi) =-\frac{m}2 + c(\pi)$. One can show that this case gives all the terms contributing to a possible limit of $\tau^{(1)}$ and $\tau^{(2)}$.

To study the quantity $q(\pi)$, we introduce a graph $\mcal {GDC}(T^\pi)$ whose topological properties govern the quantity $q(\pi)$. $\mcal {GDC}(T^\pi)$ is called the graph of deterministic components of $T^\pi$ (Definition \ref{Def:GCC}). It contains $T^\pi_X$ as a subgraph and is obtained from $T^\pi$ by replacing each connected component of $T^\pi_A$ by a single vertex and by connecting these vertices to $T^\pi_X$, see (a) in Figure \ref{Fig:Table2}. We also denote by $\overline{\mcal {GDC}}$ the graph obtained from $\mcal {GDC}(T^\pi)$ by forgetting the multiplicity of edges. These constructions come from \cite[Section 3.7.1]{Mal11} and this step is a particular instance of the asymptotic traffic independence theorem from \cite{Mal20}. 

\begin{figure}[!t]
\centering%
\subfigure[The graph of deterministic components $\mcal {GDC}(T^\pi)$ of the graph $T^\pi$ of Figure \ref{Fig:Table} (b): the left-most subgraph is of double unicyclic type, the right-most is of 2-4 tree type.]{\includegraphics{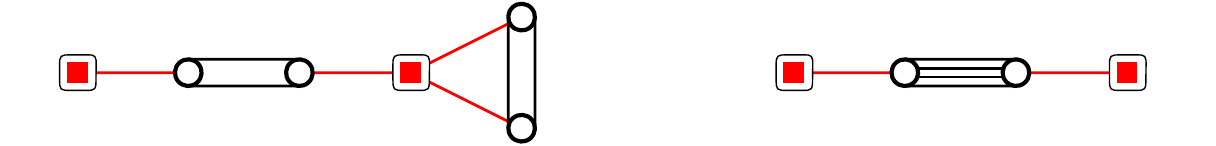}}\\
\subfigure[The graph $\overline{\mcal {GDC}}(T^\pi)$ obtained by forgetting the edge multiplicity.]{\includegraphics{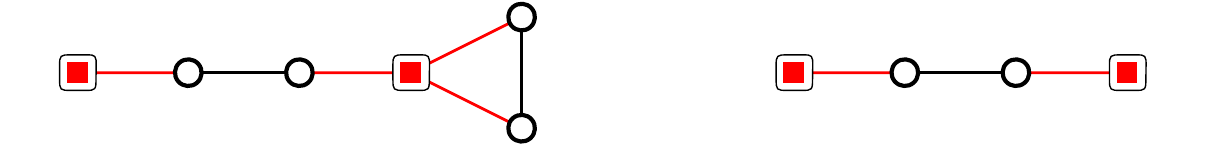}}\\
\caption{Main graphs used for the proof of main theorems.}
  \label{Fig:Table2}
\end{figure}

\subsubsection{Asymptotics: Double trees, double unicyclic and 2-4 tree types}
 In particular, the condition $q(\pi)=n$ implies that the partitions $\pi$ for which $\alpha^{(1)}(T^\pi)$ contributes in the limit are those such that $T_X^\pi$ is a forest of \emph{double trees}, that is a graph whose edges are of multiplicity 2 that becomes a forest when multiplicity is forgotten. This remark is important since double trees (rooted and embedded in the plane) are equivalent to non-crossing pairings, see \cite[\S 1.1.1]{Gui06}. 
 We then obtain 
	\eqa\label{Eq:Not2bis}
		\tau^{(1)} = \sum_{\substack{ \pi \in \mcal P(V) \\ q(\pi) =n }}\omega^{(1)}_X(\pi) \omega_A(\pi) +o(1).
	\qea

  For the second-order injective statistic, the centeredness of $Z_1\etc Z_n$ and the properties of Wigner matrices imply that $\alpha^{(2)}(\pi)$ tends to zero if at least one connected component $S$ of $T_X^\pi$ is a double tree. Using the special form of $T$, we will prove 
that: whenever $\omega^{(2)}_X(\pi)\neq0$ then $q(\pi)\leq0$. Furthermore one has that $q(\pi)= 0$ if and only if each connected component $S$ of $T^\pi$ satisfies:
  \begin{itemize}
  	\item  either the graph of deterministic components $\overline{\mcal {GDC}}(S)$ of $S$ is a unicyclic graph (removing two edges of $\mcal {GDC}(S)$ disconnects the graph) and all the edges of $S_X$ (associated to Wigner matrices) have multiplicity two (see the leftmost component of $T^\pi$ in Figure \ref{Fig:Table}),
	\item  either $\overline{\mcal {GDC}}(S)$ is a tree (removing one edge disconnects the graph) and all the edges of $S_X$ have multiplicity 2  but one which has multiplicity 4 (see the rightmost component of $T^\pi$ in Figure \ref{Fig:Table}).
\end{itemize}
We say that $\pi$ is \emph{valid} whenever $\overline{\mcal {GDC}(T^\pi)}$ satisfies the above properties. The core of the proof will be to show that
  	\eqa
		\label{Eq:ConclusionPartI}
		\tau^{(2)} = \sum_{ \substack{\pi \in \mcal P(V) \\  \mrm{ \ valid}}} \omega^{(2)}_X(\pi) \omega_A(\pi) +o(1).
	\qea

In order to show \eqref{Eq:ConclusionPartI}, one has to consider the case of a partition $\pi$ such that $\mathfrak f(\pi)>2c(\pi)$ (such a partition is not valid). The number of additional leaves $ {\mathfrak f(\pi)}-2c(\pi)$ is then controlled by the number of cycles of $\overline{\mcal {GDC}}$. Quantitatively, this is measured thanks to the \emph{pruning} $\tilde {\mcal {GDC}}(T^\pi)$ (Definition \ref{Def:Pruning}) of the graph of deterministic components. It is obtained by inductively erasing the leaves of $\overline{\mcal {GDC}}$ until the graph is deleafed, see the rightmost graph in Figure \ref{Fig:TableAux}. This reasoning allows to prove that  if $\mathfrak f(\pi)>2c(\pi)$ then $q(\pi)<0$. We can then conclude that the contribution from non-valid partitions vanishes at infinity.

\subsection{Expansion in terms of graphs and separation of contributions}

We first write $\tau^{(1)}$ and $\tau^{(2)}$ using graph notations. Unless explicitly mentioned, graphs are directed, they can be disconnected, and admit possibly loops and  multiple edges. Formally, $V$ is a set and $E$ is a multiset (elements appear with a multiplicity) of elements of $V^2$. We consider indeterminates $x_1\etc x_m, a_1\etc a_m$.

\begin{definition}	
\label{Def:LabGraph}
A labeled graph is a triple $T=(V,E,\gamma)$, where $(V,E)$ is a finite graph and $\gamma$ is a labeling map from $E$ to a subset of $\{1\etc m\}$.
\end{definition}

Below, labeled graphs are given with a partition $E=E_X \sqcup E_A$ of the edge set. Accordingly an edge $e$ in $E_X$ (resp. in $E_A$) such that $\gamma(e)=k$ is associated to the indeterminate $x_k$ (resp. $a_k$). For any $j\in [n]$, we first represent $\Tr  \, M_j$ by a labeled graph $T_j =(V_j,E_j,\gamma_j)$ as follows. The directed graph $(V_j,E_j)$ consists of a simple oriented cycle with $2p_j$ edges, see (a) in Figure \ref{Fig:Table}: we have $V_j = V_j^\circ \sqcup V_j^{\bullet}$, with
	\eq
		V_j^\circ=\{m_j+1\etc m_j+p_j\}, &  V_j^\bullet=\{(m_j+1)'\etc (m_j+p_j)'\}.
	\qe
and $E_j = E_{j,X}\sqcup E_{j,A}$, with
\begin{multline*}
E_{j,X} = \Big\{ e_{m_j+1,x} = \big((m_j+1)',m_j+1\big), \dots,\\
e_{m_j+p_j,x} = (m_j+p_j)',m_j+p_j) \Big\},
\end{multline*}
(representing edges from a vertex of $V_j^\bullet$ to a vertex of $V_j^\circ$),
\begin{multline*}
		E_{j,A} = \Big\{e_{m_j+1,a}=\big( m_j+2,(m_j+1)'\big),\dots, \\
		 e_{m_j+p_j,a}=(m_j+1, (m_j+p_j)')  \Big\},
\end{multline*}
(representing edges from a vertex of $V_j^\circ$ to a vertex of $V_j^\bullet$).
We assign to each edge a label, by means of a map $\gamma_j:E_j \to \{m_j+1, m_j+p_j\}$ given by $\gamma_j(e_{k,x}) = \gamma_j(e_{k,a}) =k$. This indicates that the edge $e_{k,x}$ is associated to the indeterminate $x_k$ and $e_{k,a}$ is associated to the indeterminate  $a_k$. 

\begin{definition}\label{Def:RetTrace}
\begin{enumerate}
	\item For any labeled graph $T=(V,E,\gamma)$ and for any map $\psi:V\to [N]$, we denote 
	\eq
		 { r(T,\psi)}  
		& = &  \prod_{\substack{e\in E_X}} X_{\gamma(e)}\big( \psi( \mrm{trg} \ e ),\psi( \mrm{src}\, e ) \big) \\
		& &  \times \prod_{\substack{e'\in E_A }}  A_{\gamma(e')}\big( \psi(  \mrm{trg}\, e' ),\psi( \mrm{src} \, e' ) \big),
	\qe
with $\mrm{src}\, e$ and $\mrm{trg}\, e$ denoting the source and the target, respectively, of the edge $e$. 
	\item  For any labeled graph $T=(V,E,\gamma)$, we denote
	$$\Tr [T] = \sum_{ \substack{ \psi: V \to [N]}} r(T,\psi),$$
	and call it the (unnormalized) \emph{trace} of the labeled graph $T$.
\end{enumerate}
\end{definition}

With the above definition we have $\Tr [T_j] = \Tr[M_j]$, and so 
	\eq
		\tau^{(1)} &:=& N^{-n} \E \Big[ \prod_{j=1}^n \Tr \, M_j \Big] = N^{-n} \E \Big[ \prod_{j=1}^n \Tr \, T_j \Big],\\
	\tau^{(2)}& :=&  \E\bigg[ \prod_{j=1}^n \Big( \Tr \,M_j  - \E\big[ \Tr \, M_j \big] \Big) \bigg]= \E\bigg[ \prod_{j=1}^n \Big( \Tr \, T_j  - \E\big[ \Tr \, T_j \big] \Big) \bigg].
	\qe

We denote by $T = T_1 \sqcup  \dots \sqcup T_n = ( V,E, \gamma)$ the labeled graph obtained as the disjoint union of the $n$ graphs. Seeing each $T_i$ as a subgraph of $T$, the vertex set $V$ is $V_1\sqcup \dots \sqcup V_n$, the edge set $E$ is $E_1\sqcup \dots \sqcup E_n$. Then the map $\gamma:E\to [m]$ is defined by $\gamma(e) = \gamma_j(e)$ whenever $e\in E_j$.

We can then write $\tau^{(1)}$ and $\tau^{(2)}$ as functions of $T$:
	\eqa
	\label{Eq:Proof:Step0}
		\tau^{(1)} &=&  N^{-n} \E \Big[ \Tr \, T \Big] = N^{-n}\sum_{\psi : V \to [N]} \E \Big[  r(\sqcup_{j\in [n]} T_{j},\psi)\Big],
	\\\label{Eq:Proof:Step1}
		\tau^{(2)} &=& \sum_{\psi : V \to [N]} \E \bigg[ \prod_{j=1}^n\Big( r(T_j,\psi_{|V_j}) - \E \big[ r(T_j,\psi_{|V_j}) \big] \Big) \bigg].
	\qea

For each $j\in [n]$ we define the subgraph $T_{j,X}=(V_{j}, E_{j,X}, \gamma_{j,X})$ of $T_j$ consisting of the edges $e_{p_{j-1}+1} \etc e_{p_j}$ associated to Wigner matrices. 
The vertex set of $T_{j,X}$ is still $V_j$ and the labelling map $\gamma_{j,X}$ is the restriction of $\gamma_j$ to $E_{j,x}$. We also introduce the labeled subgraph  $T_A = (V,E_A,\gamma_A)$ of $T$ consisting of the edges $e_1' \etc e_{p_n}'$. All these labeled graphs consist of disjoint simple edges. 

For any map $\psi:V\to [N]$, any $j=1\etc n$ and any $J\subset [n]$, we denote by $\psi_j$ the restriction of $\psi$ to $V_j$ and by $\psi_J$ its restriction to $\sqcup_{i\in J} V_i$. Since $r(T_A,\psi)$ is a deterministic quantity, we have
	\eq
		\E \Big[  r(\sqcup_{j\in [n]} T_{j},\psi)\Big]& = &   \E \Big[ r(\sqcup_{j\in [n]} T_{j,X},\psi) \Big] \times r(T_A,\psi)
	\qe
and
	\eq
		\lefteqn{\E \bigg[ \prod_{j=1}^n\Big( r(T_j,\psi_j) - \E \big[ r(T_j,\psi_j) \big] \Big) \bigg]  }\\
		& = & \E \Big[ \prod_{j=1}^n\Big( r(T_{j,X}\psi_j) - \E \big[ r(T_{j,X},\psi_j) \big] \Big)  \Big]\times r(T_A,\psi)\\
		& = & \kern-3pt\sum_{J\subset[n]} \E \Big[ r(\sqcup_{j\in J} T_{j,X},\psi_J) \Big] \times (-1)^{n-|J|}\prod_{j\notin J}\E \Big[ r( T_{j,X},\psi_j) \Big]  \times r(T_A,\psi).	
	\qe
With \eqref{Eq:Proof:Step0} and \eqref{Eq:Proof:Step1} this gives a first expression for $\tau^{(1)}$ and $\tau^{(2)}$ where we separate the terms from Wigner and deterministic matrices.

\subsection{Regrouping terms and good decomposition}

For any map $\psi:V \to [N]$, denote by $\ker \psi$ the partition such that $v\sim_{\ker \psi} w$ whenever $\psi(v)=\psi(w)$, i.e. two vertices are in a same block if and only if $\psi$ attributes the same value for both of them. By permutation invariance of the Wigner matrices,  the value of $\E \big[ r( \sqcup_{j\in J} T_{j,X},\psi_j) \big]$ depends on $\psi$ only through the restriction of $\ker \psi$ on $\sqcup_{j\in J}V_j$. For any partition $\pi$ of $V$ and any $J\subset[N]$, we denote by $\pi_J$ the restriction of $\pi$ on $\sqcup_{j\in J}V_j$ and by $ R( \sqcup_{j\in J} T_{j,X},\pi_J)$ the common value of $\E \big[ r( \sqcup_{j\in J} T_{j,X},\psi_J) \big]$ for any $\psi$ such that $\ker \psi=\pi$. 
We then deduce from \eqref{Eq:Proof:Step0} and \eqref{Eq:Proof:Step1} 
	\eqa
	\label{eq:QuotientPart:0}
		 \tau^{(1)}&=&  \sum_{\pi \in \mcal P(V)} N^{-n} \beta^{(1)}_X(\pi) \times \beta_A(\pi),\\
		 \label{eq:QuotientPart:1}
		 \tau^{(2)}&=&  \sum_{\pi \in \mcal P(V)}  \beta^{(2)}_X(\pi) \times \beta_A(\pi),
	\qea
 where
	\eq
		\beta^{(1)}_X(\pi)& = & R(\sqcup_{j\in [n]} T_{j,X},\pi), \: \beta_A(\pi)  =  \sum_{ \substack{ \psi: V \to [N] \\ \mrm{ s.t. \, } \ker \psi=\pi}} r(T_A,\psi),\\
		\beta^{(2)}_X(\pi)& = & \sum_{J\subset[n]}   (-1)^{n-|J|}R(\sqcup_{j\in J} T_{j,X},\pi_J)  \times \prod_{j\notin J}R( T_{j,X}, \pi_j).
	\qe 

\subsubsection{Contribution from Wigner matrices}
By the definition of the function $R$ and due to the $\frac 1{\sqrt N}$-normalization of Wigner matrices, it is clear that $\omega^{(i)}_X(\pi) := N^{\frac m 2} \beta^{(i)}_X(\pi)$ is bounded. We re-write below this contribution in an explicit way.

\begin{definition} \label{Def:InjTrace}
For any labeled graph $S=(W,F,\delta)$ and any partition $\pi$ of $W$, we define by $S^\pi=(W^\pi, F^\pi, \delta^\pi)$ the labeled graph obtained by identifying vertices of $S$ that belong to a same block of $\pi$. An edge $e=(v,w)$ of $S$ becomes an edge $e_\pi=(B_v,B_w)$ of $S^\pi$, where $B_v$ and $B_w$ are respectively the blocks containing $v$ and $w$. The label of $e_\pi$ is the label of $e$, namely $\delta^\pi(e_\pi) = \delta(e)$. We say that $S^\pi$ is a \emph{quotient} of $S$. 
\end{definition}

Note that for $j=1\etc n$ and any $\pi\in \mcal P(V_j)$ one has $R(T_{j,X},\pi)  = R(T_{j,X}^{\pi}, \mbf 0)$, where $\mbf 0$ is the partition consisting of singletons only. 

\begin{definition}\label{Def:OmegaX}
Let $T^\pi$ be a quotient graph of $T=T_1\sqcup \cdots \sqcup T_n$. Denote by $T_j^\pi=(V^\pi_j,E^\pi_j,\gamma^\pi_j)$ the quotient of $T_j$ by $\pi$ for any $j\in [n]$. We define the weights  associated to the Wigner matrices by
	\eqa\nonumber
			\omega^{(1)}_X(\pi) & =  &\E\Big[ \prod_{\substack{ e\in E^\pi_{j,x} \\ \forall j\in [n]}} \sqrt N X_{\gamma_j^\pi(e)}\big( \psi_0(  \mrm{trg}\, e ),\psi_0( \mrm{src} \, e) \big)\Big] \\
	 		\omega^{(2)}_X(\pi)& = &\sum_{J\subset [n]} \E\Big[ \prod_{\substack{ e\in E^\pi_{j,x} \\ \forall j\in J}} \sqrt N X_{\gamma_j^\pi(e)}\big( \psi_0(  \mrm{trg}\, e ),\psi_0( \mrm{src} \, e) \big)\Big] \nonumber\\
			& & \ \ \ \times (-1)^{n-|J|} \prod_{j\notin J}\E\Big[\prod_{e\in E^\pi_{j,x}}  \sqrt N X_{\gamma^\pi_j(e)}\big( \psi_0(  \mrm{trg}\, e ),\psi_0( \mrm{src} \, e) \big)   \Big],\label{Eq:OmegaX}
	\qea
for any choice of injective map $\psi_0:V_j\to [N]$ (the value is independent of this choice).
\end{definition}

\begin{example} Let $n=4$ and $T$ and $\pi$ be as Figure \ref{Fig:Table}. We denote by $z_j$ a random variable distributed as the $(1,2)$ entry of the (unnormalized) Wigner matrix $\sqrt N X_j$. For any $J\subset [n]$ we denote by $\omega_X(\pi)[J]$ the corresponding summand in  Equation \eqref{Eq:OmegaX}. For $J=\{1,2,3,4\}$, we have
	$$\omega_X(\pi)[J] = \E[ z_2\overline{z_3}]\E[z_1z_4] \E[z_5z_7\overline{z_6z_8}].$$
We recall that when denoting the Wigner matrices $X_1\etc X_m$ we allow possible repetition of a same matrix, and so this term is possibly nonzero only when $X_2=X_3$, $X_1=X_4$, and some repetitions occur among the matrices $X_5 \etc X_8$. Note also that we can indifferently write $\E[ z_3\overline{z_2}]$ instead of $\E[ z_2\overline{z_3}]$ since these quantities are equal by complex conjugate invariance of Wigner matrices entries. 
For $J=\{1,2\}, \{1,2,3\}, \{1,2,4\}$, we have in each case
	$$\omega_X(\pi)[J] = \E[ z_2\overline{z_3}]\E[z_1z_4] \E[z_5\overline{z_6}]\E[z_7\overline{z_8}].$$
Otherwise, when $J$ does not contain both $1$ and $2$, then 
	$$\omega_X(\pi)[J] = \E[ z_2\overline{z_3}]\E[z_1]\E[z_4] \times \dots = 0$$
 by centeredness of the entries.

\end{example}

\subsubsection{Contribution from deterministic matrices}

\begin{definition}\label{Def:TEC}
For any labeled graph $S$ with vertex set $W$, the (unnormalized) \emph{injective trace} of the labeled graph $S$ is
	$$\Tr^0 [S] = \sum_{ \substack{ \psi: W \to [N] \\ \mrm{injective}}} r(S,\psi),$$
where $r(S,\psi)$ is defined in Definition \ref{Def:RetTrace}.
\end{definition}
We also need the following definition.
\begin{definition}\label{Def:TEC2}
\begin{enumerate}
	\item A cutting edge of a graph is an edge whose removal increases the number of connected components.
	\item A two-edge connected graph is a connected graph with no cutting edge. Similarly a two-edge connected component of a graph is a maximal connected sub-graph which is two-edge connected.
	\item The forest of two-edge connected components of a graph $S$ is the graph $\mcal T\hspace{-2pt}ec(S)$ whose vertices are the two-edge connected components of $S$ and whose edges are the cutting edges of $S$, making links between the components that contain the source and the target of the edge.
	\item A trivial component of $\mcal T\hspace{-2pt}ec(S)$ is a component consisting of a single vertex. We denote by  $\mathfrak f( S)$ the number of leaves of the forest of two-edge connected components $\mcal T\hspace{-2pt}ec(S)$, with the convention that a trivial component has two leaves.
	\end{enumerate}
\end{definition}
Lastly we define the weight associated to the deterministic matrices.
\begin{definition}
 Let $\pi$ be a partition of the vertices of $T$. We define the weight $\omega_A(\pi)$  associated to the deterministic matrices by
	\eqa\label{Eq:OmegaA}
		\omega_A(\pi) & = &   {N^{-\frac{\mathfrak f(T^\pi_A)}2}}\Tr^0\big[T^\pi_A\big].
	\qea
\end{definition}

\begin{example}Let $T$ and $\pi$ be as Figure \ref{Fig:Table}. The leftmost graph $T_A^\pi$ of (d) in Figure \ref{Fig:Table} has 4 connected components, all of them are two edge-connected. Hence by convention one has that $\frac{\mathfrak f(T^\pi_A)}2=4$. Then one has
	\eq
		\omega_A(\pi) &  =  &   N^{-4} \sum_{\substack{i_1\etc i_8\in [N]\\\mrm{pairwise \ distinct}}} 
		A_2(i_1,i_1)  \times
		A_1(i_2,i_3) A_3(i_3,i_4)A_4(i_2,i_4)
		\\
		& & \ \ \ \times  A_5(i_5,i_5)A_7(i_5,i_5)
		\times A_6(i_6,i_6)  A_8(i_6,i_6)
	\qe
\end{example}

\begin{lemma}For any $\pi\in \mcal P(V)$, $\omega_A(\pi)$ is bounded uniformly in $N$. 
\end{lemma}
\begin{proof}
We have the relation
	$$\Tr^0[T_A^\pi] = \sum_{\pi' \in \mcal P(V^\pi)} \mrm{Mob}(0,\pi') \Tr[T_A^{\pi'}],$$
where $ \mrm{Mob}$ is the M\"obius function of the poset of partitions of the vertex set of $T_A^\pi$, see \cite{Mal20}. By \cite{MS11}, for any labeled graph $T_A^{\pi'}$ we have the bound
	$$ \big| \Tr [T_A^{\pi'}] \big| \leq N^{\frac{\mathfrak f(T_A^{\pi'})}2} \times \prod_{j=1}^n\| A_j\|.$$
Note moreover that $\mathfrak f(T_A^\pi) \geq \mathfrak f(T_A^{\pi'})$ for any $\pi'\in \mcal P(V^\pi)$, i.e. the number of leaves of the forest of two-edge connected components cannot increase by taking a quotient. Hence $\Tr[T_A^{\pi'}] = O \big( N^{  \frac{ \mathfrak f( T_A^\pi)}2}\big)$ for any $\pi'\in \mcal P(V^\pi)$, and so we get as well $\Tr^0[T_A^\pi] = O \big( N^{  \frac{ \mathfrak f( T_A^\pi)}2}\big)$. The proof then follows from (\ref{Eq:OmegaA}).
\end{proof}

Recalling from Eq. (\ref{def:q}) that  $q(\pi) = -\frac{m}2+\frac{\mathfrak f(T_A^\pi)}2$, we have finally obtained
	\eqa\label{Eq:Step0}
		\tau^{(1)} &=& \sum_{ \pi \in \mcal P(V)} N^{q(\pi)-n}  \omega^{(1)}_X(\pi) \times \omega_A(\pi),\\
	\label{Eq:Step1}
		\tau^{(2)} &=& \sum_{ \pi \in \mcal P(V)} N^{q(\pi)}  \omega^{(2)}_X(\pi) \times \omega_A(\pi),
	\qea
where we recall that $\omega^{(1)}_X$ and $\omega^{(2)}_X$ are defined in \eqref{Eq:OmegaX}, $\omega_A$ is defined in \eqref{Eq:OmegaA}. Note that $m$ equals $|E_X|$ the total number of edges associated to Wigner matrices in $T$ (or equivalently in $T^\pi$).

\subsection{The topological analysis} 
In this subsection we identify the partitions that contribute to $\tau^{(1)}$ and $\tau^{(2)}$. To describe the connected components that contribute to the limit of these statistics, we need the following Definitions \ref{Def:GCC} and \ref{Def:DUG}.

Fix $\pi$ in $\mcal P(V)$ and denote by $V^\pi$ the vertex set of $T^\pi$ and $E^\pi$ its multi-set of edges. We first analyze in more detail the quantity $q(\pi)$. 
\begin{definition}\label{Def:GCC} The graph of deterministic components of $T^\pi$ is the undirected graph  $\mcal {GDC}(T^\pi)= (\mcal V, \mcal E)$, where
\begin{itemize}
	\item the vertex set $\mcal V$ consists of the disjoint union of the vertex set $V^\pi$ of $T_X^\pi$ and of the set $  C_{A}$ of connected components of $T^\pi_A$ (we will in the following call the elements of $C_A$ as \textit{deterministic components}),
	\item the edge set $\mcal E$ consists of the disjoint union of $E^\pi_{X}$ (i.e. the set of edges of $T_X^\pi$) and of the set of pairs $(v,C)$, $v\in V^\pi$, $C\in C_A$ such that $v \in C$ in the graph $T^\pi$. 
\end{itemize}

We also denote $\overline{\mcal {GDC}}(T^\pi) (\mcal V , \overline{\mcal E} )$ the graph obtained from $\mcal {GDC}(T)$ by forgetting the multiplicity of its edges, and assuming that this multiplicity is one for each edge. 
\end{definition}

See examples (e) and (f) in Figure \ref{Fig:Table}. When the quotient graph $T^\pi$ is fixed without ambiguity, we write $\mcal {GDC}$ as a shortcut for $\mcal {GDC}(T^\pi)$. By definition, the number of vertices of ${\mcal {GDC}} $ is $|\mcal V| = |V^\pi| + |C_{A}|$. Denote by $|\overline{E^\pi_X}|$ the number of edges of $T^\pi_X$ counted without multiplicity. We see that the number of edges without counting multiplicity is $|\overline{\mcal E}| = |\overline{E^\pi_X}| + |V^\pi|$, since each vertex $v$ of $T^\pi$ is connected exactly to one deterministic component. We then write 
	\eqa
		q(\pi) & = & -\frac{|E_X|}2+\frac{\mathfrak f(T^\pi_A)}2\nonumber\\
		& =&   \underbrace{\big( |\overline{E^\pi_X}| - \frac{|E_X|}2 \big)}_{:=q_1}  +  \underbrace{\big( |\mcal V| - |\bar{\mcal E}| \big)}_{:=q_2}  +  \underbrace{\Big( \frac{\mathfrak f (T^\pi_A)}2 - |C_{A}| \Big)}_{:=q_2'}.\label{3TermsTopoExp}
	\qea
Note that $q_1$ and $q_2'$ are half integers, whereas $q_2$ is an integer. All the quantities involved are implicitly functions of $\pi$, and are additive with respect to the different connected components $\mcal {GDC}_i$, $i=1\etc c$, of $\mcal {GDC}$. We denote, for each $i=1\etc c$, by  $q_{1,i}$, $q_{2,i}$, and $q_{2,i}'$ the version of $q_1$, $q_2$, and $q_2'$, respectively, defined for the $i$-th component $\mcal {GDC}_i$.

We here state two lemmas that we use in the rest of the section.
\begin{lemma}\label{Lem:VerticesEdges} Let $\mcal G=(\mcal V, \mcal E)$ be a finite connected graph. Then $|\mcal V| - |\mcal E|  \leq 1$ with equality if and only if $\mcal G$ is a tree. 
\end{lemma}

The second one is referred to as a \emph{parity argument}.

\begin{lemma}\label{Lem:ParityArgument}
Let $S$ be the quotient of a union of simple cycles. Let $\bar e$ be a group of twin edges in $S$. If the removal of the edges of  $\bar e$ disconnects the graph $S$, then the multiplicity of the edges of $S$ coming from each cycle is an even number, with an equal number of edges in one direction and in the other direction. 
\end{lemma}

\begin{proof}
Assume that a cycle $C$ has $m\geq 1$ twin edges in the group of edges associated to $\bar e$. We show the lemma by induction on $m$. 

We first prove that necessarily $m>1$. Assume, for a contradiction, that $C$ has a one single edge $e_0$ that represents $\bar e$. Let $C\setminus {e_0}$ be the graph obtained from $C$ by removing $e_0$. Since $C$ is a cycle, $C\setminus {e_0}$ is connected, and so in every quotient of $C\setminus {e_0}$ the image of the source and target of $e_0$ belong to a same connected component. Let $S\setminus{\bar e}$ denote the graph obtained from $S$ by removing all the edges representing $\bar e$. Since in  $S\setminus{\bar e}$ there is a subgraph which is a quotient of $C\setminus{e_0}$, then the source and target of the edge of $\bar e$ belong to a same connected component. But this is in contradiction with the assumption that the removal of $\bar e$ disconnects $S$. Hence the multiplicity $m$ is at least equal to 2. 

We now assume that $m$ is larger than 2. We consider a closed walk on $S$, given by the image of the cycle $C$ and starting at an edge $e_0$ in $C$ representing $\bar e$. Let $e_1$ be the edge of $C$ that is the first representant of $\bar e$ that the walk meets after $e_0$. Necessarily in the quotient graph, the source of $e_0$ is the target  of $e_1$ and vice versa: indeed, otherwise one sees (as in the previous paragraph) that removing the edges of $\bar e$ in  $S$ would not disconnect the graph. We now denote by $C'$ the graph obtained  from $C$ by identifying the source of $e_0$ with the target of $e_1$, and deleting $e_0, e_1$, all the edges inbetween them and all vertices that stay isolated after this process. Hence $C'$ is a simple cycle such that of smaller size and has $m-2$ edges representing $e$. By induction, $C$ has an equal number of edges representing $\bar e$ in both directions, which concludes the proof of Lemma \ref{Lem:ParityArgument}.
\end{proof}

We can now state the main result of this subsection, thanks to the following definition.
\begin{definition}\label{Def:DUG}
\begin{itemize}
	\item The $i$-th component of $T^\pi$ is of \emph{double tree type} whenever $(q_{1,i},q_{2,i})=(0,1)$, which means that the edges of $\mcal {GDC}_i$ associated to Wigner matrices have multiplicity two and that $\overline{\mcal {GDC}}_i$ is a tree. 
	\item  The $i$-th component of $T^\pi$ is of \emph{double unicyclic type} whenever $(q_{1,i},q_{2,i})=(0,0)$, which means that the edges of $\mcal {GDC}_i$ associated to Wigner matrices have multiplicity two and that $\overline{\mcal {GDC}}_i$ is a graph with a unique simple cycle. 

	\item The $i$-th component of $T^\pi$ is of \emph{2-4 tree type} whenever $(q_{1,i},q_{2,i})\ab =(-1,1)$, which means that $\overline{\mcal {GDC}}_i$ is a tree and that the edges of $\mcal {GDC}_i$ associated to Wigner matrices have multiplicity two, except for one group of edges of multiplicity four.
\end{itemize}
We denote by $\mcal{D}\mcal{T}_c$ the set of partitions $\pi$ of $V$ such that $T^\pi$ has $c\in [n]$ connected components and all are of double tree type. We denote by $\DUFT$ the set of partitions such that the components are either of double unicyclic or 2-4 tree type.
\end{definition}

 Here is now the main result in our analysis of the $\tau$-functions, which computes the leading order in their asymptotic expansion. Recall that $n$ is the number of connected components of $T$.
 
\begin{proposition}\label{prop:topanalys}
One has the asymptotic large $N$ expansion
	\eqa\label{Eq:ConclusionPartIbis0}
		\tau^{(1)} = \sum_{ \pi \in \mcal D\mcal T_{n}}\omega^{(1)}_X(\pi) \times \omega_A(\pi) +o(1),\\
	\label{Eq:ConclusionPartIbis}
		\tau^{(2)} = \sum_{ \pi \in \DUFT}\omega^{(2)}_X(\pi) \times \omega_A(\pi) +o(1).
	\qea
\end{proposition}
\begin{remark}The asymptotics for $\tau_1$ are the same as in the Gaussian case. The first statement is thus a universality statement.
\end{remark}

The proof of Proposition \ref{prop:topanalys} relies on the following two lemmas, where we recall that the notations for $q(\pi) = q_1+q_2+q_2'$ is from \eqref{3TermsTopoExp}.

\begin{lemma}\label{Lem:Dominating} Let $\pi\in \mcal P(V)$ and denote by $c=c(\pi)$ the number of connected components of $T^\pi$. If $\omega_X^{(1)}\neq 0$, then $q_1+q_2\leq c$ with equality if and only if $\pi\in \DT_c$. Moreover if $\omega_X^{(2)}\neq 0$ then $q_{2,i}+q'_{2,i}\leq 0$ with equality if and only if $\pi \in \DUFT$.
\end{lemma}
The quantity $q'_{2,i}$ is taken under consideration in the next Lemma.
\begin{lemma}\label{Lem:ErrorTerms} 
\begin{enumerate}
	\item If the  multiplicity in the $i$-th component of $T^\pi$ of edges labeled by Wigner matrices is even, then $q_{2,i}'=0$.
	\item If $q_{2,i}\leq 0$, i.e. $\overline{\mcal {GDC}}_i$ is not a tree, then  $q_{2,i}+q_{2,i}'\leq 0$. 
\end{enumerate}
\end{lemma}

\begin{proof}[Proof of Proposition \ref{prop:topanalys}] Assume that Lemmas \ref{Lem:Dominating} and \ref{Lem:ErrorTerms} hold true and let $\pi\in \mcal P(V)$.

Assume that $\omega^{(1)}_X(\pi)\neq 0$. Then one has that $q_{1,i}\leq 0$. Either $\overline{\mcal {GDC}}_i$ is a tree, and the parity argument (Lemma \ref{Lem:ParityArgument}) implies that the number of edges labeled by Wigner matrices is even; so the first parts of the lemmas imply  $q(\pi) - n = q_{1}+q_{2}+q_{2}'-n=q_{1}+q_{2}-n \leq 0$  with equality only if $\pi \in \DT_n$. 
Either $\overline{\mcal {GDC}}_i$ is not a tree, and the second part of Lemma \ref{Lem:ErrorTerms} implies $q_i(\pi) - 1 = q_{1,i}+q_{2,i}+q'_{2,i}-1\leq q_{1,i}-1\leq -1$. Hence in \eqref{Eq:Step0} the only $\pi$ that contribute for $\tau^{(1)}$ in the limit are those such that $T^\pi\in  \DT_n$, which proves of \eqref{Eq:ConclusionPartIbis0}.

Assume now that $\omega^{(2)}_X(\pi)\neq 0$. If $\overline{\mcal {GDC}}_i$ is a tree, the parity argument and the first part of Lemma \ref{Lem:ErrorTerms} imply again $q_{2,i}'=0$, and the second part of Lemma \ref{Lem:Dominating} implies that  $q_i(\pi)  = q_{1,i}+q_{2,i}+0\leq 0$ with equality whenever the $i$-th component of $T^\pi$ is of 2-4 tree type. 
Assume now $\overline{\mcal {GCD}_i}$ is not a tree, i.e. $q_{2,i}\leq 0$. By  the second part of Lemma \ref{Lem:ErrorTerms}, one has that $q_i(\pi)  =q_{1,i}+q_{2,i}+q_{2,i}'\leq 0$. If $q_i(\pi) =0$,  the multiplicity of edges labeled by Wigner matrices is 2. The first part  of Lemma \ref{Lem:ErrorTerms} hence implies that $q_{2,i}'=0$. Hence, when $\mcal {GCD}_i$ is not a tree, $q_i(\pi)=0$  if and only if the $i$-th component  is of double unicyclic type. We thus have shown that the partitions $\pi$ that contribute to $\tau^{(2)}$ in  \eqref{Eq:Step1} are those such that $\pi\in \DUFT$, which proves \eqref{Eq:ConclusionPartIbis} and concludes the proof of Lemma \ref{prop:topanalys}.
\end{proof}

\subsubsection{ Proof of Lemma 23}

We turn to the proof of Lemma \ref{Lem:Dominating}.

The first function $q_1$ is a linear combination of the numbers of edges of $T^\pi$ labeled $X$ when counting and not counting the multiplicity. If $T^\pi$ has an edge of multiplicity one then for any $J\subset [n]$ so has either $\sqcup_{j\in J} T^\pi_J$ or one of the $T^\pi_j$, $j\notin J$; hence by independence and centeredness of the entries of Wigner matrices, and by Formula  \eqref{Eq:OmegaX} we get $\omega_X^{(1)}(\pi) = \omega_X^{(2)}(\pi)=0$. So we can assume that the multiplicity of each edge of $T^\pi$ labeled by a Wigner matrix is at least 2 and we can restrict to $\pi$ with $ q_1\leq 0.$

Lemma \ref{Lem:VerticesEdges} applied component-wise implies $q_{2,i}\leq 1$ for any $i=1\etc c$ with equality whenever $\overline{\mcal {GDC}_i}$ is a tree. Assuming there is no edge of multiplicity one in $T^\pi_X$, the possible maximal order of $N^{q_{1,i}+q_{2,i}}$ given by the $i$-th connected component of $T^\pi$ is $N$ when $q_{1,i}=0$ and $q_{2,i}=1$. This means that the edges of $\mcal {GDC}_i$ associated to Wigner matrices have multiplicity two and that $\overline{\mcal {GDC}}_i$ is a tree, i.e. the components of $T^\pi$ are of \emph{double tree type}. 

We have proved the first part of Lemma \ref{Lem:Dominating}: when $\omega^{(1)}_X(\pi)\neq 0$ (and in particular $ q_1\leq 0$) then $q_1+q_2\leq c$ with equality whenever $\pi \in  \DT_c$.

For the second part of the lemma, we use further arguments. We define below the property of \emph{$X$-connectedness}. Recall that $T^\pi_j$, $j=1\etc n$, denote the quotients of the cycles forming $T$, that we see as subgraphs of $T^\pi$. 

\begin{definition} 
\begin{itemize}
	\item We say that two edges $e$ and $e'$ of $T_X^\pi$ are \emph{twin} in $\mcal {GDC}(T^\pi)$ whenever they share the same pair of vertices: $\mrm{trg} \ e= \mrm{trg} \ e'$ and $ \mrm{src} \ e = \mrm{src} \ e'$, or $\mrm{trg} \ e = \mrm{src} \ e'$ and $ \mrm{src} \ e = \mrm{trg} \ e'$.
	\item We say that two graphs $T^\pi_j$ and $T^\pi_{j'}$ are \emph{$X$-connected} whenever there is an edge $e$ of $T^\pi_j$ and an edge $e'$ of $T^\pi_{j'}$ that are twin and associated to Wigner matrices. If $T^\pi_j$ and $T^\pi_j$ are not $X$-connected we say that they are \emph{$X$-disconnected}. 
\end{itemize}
\end{definition}

\begin{lemma}\label{Lem:Key}  If there exists an index $j_0$ such that  $T^\pi_{j_0}$ is $X$-disconnected from all the other graphs then $\omega^{(2)}_X(\pi)=0$. 
\end{lemma}

\begin{proof} Assume that  $T^\pi_{j_0}$ is $X$-disconnected from all the other graphs. Using (\ref{Eq:OmegaX}), one has that 
\eqa\nonumber
	 		\omega^{(2)}_X(\pi)& = &\sum_{J\subset [n], j_0\in J} \E\Big[ r\big(  \cup_{j\in J} T_j, \psi \big) \Big]\ (-1)^{n-|J|} \prod_{\substack{j\notin J}}\E\Big[ r\big(   T_{j}, \psi_i\big) \Big]\cr
			&&+\sum_{J\subset [n], j_0\notin J} \E\Big[ r\big(  \cup_{j\in J} T_j, \psi \big) \Big]\ (-1)^{n-|J|} \prod_{\substack{j\notin J}}\E\Big[ r\big(   T_{j}, \psi_i\big) \Big].
\qea

 The independence of the entries of the Wigner matrices implies that we can factorize the expectation associated to the graph $T^\pi_{j_0}$ in the first sum yielding that  $\omega^{(2)}_X(\pi)=0$.
\end{proof}

\begin{lemma}\label{Lem:NoDoubleTree} For any $\pi \in \mcal P(V)$, if $T^\pi$ has a connected component of double tree type, i.e. such that $q_{1,i}=0$ and $q_{2,i}=1$, then $\omega^{(2)}_X(\pi)=0$.
\end{lemma}

\begin{proof}
Assume that a connected component $S$ of $T^\pi$ is of double tree type. Lemma \ref{Lem:ParityArgument} implies that twin edges must come from a single cycle $T_i$, hence each graph $T^\pi_i \subset S$ is $X$-disconnected from all other graphs. Lemma \ref{Lem:Key} implies that $\omega^{(2)}_X(\pi)=0$. 
\end{proof}

By Lemma \ref{Lem:ParityArgument}, computing the asymptotic of $\tau^{(2)}$ we can then assume hereafter that for any $i=1\etc c$ one has $q_{1,i}+q_{2,i}<1$. The next possible order for $N^{q_{1,i}+q_{2,i}}$ is a priori $\sqrt N$, when $q_{1,i}=-\frac 1 2$ and $q_{2,i} = 1$. But this means that $\overline{\mcal {GDC}}_i$ is a tree and there is an edge of $\mcal {GDC}_i$ labeled $X$ of multiplicity 3, all other edges  labeled $X$ being of multiplicity 2. This is not possible by Lemma \ref{Lem:ParityArgument}.
Hence if $\omega^{(2)}_X(\pi)\neq 0$ we have as claimed
\eqa \label{Eq:q1q2}
			q_{1,i}+q_{2,i}\leq 0, \ \forall \, i=1\etc c.
\qea
The two cases of equality are when $(q_{1,i},q_{2,i}) = (0,0)$ and $(-1,1)$, which corresponds to the condition $\pi \in \DUFT$ (Definition \ref{Def:DUG}). This finishes the proof of Lemma \ref{Lem:Dominating}.

\subsubsection{Proof of Lemma 24}

We now take the quantity $q_{2,i}'$ under consideration and turn to the proof of Lemma \ref{Lem:ErrorTerms}. We say that a undirected graph is an \emph{Eulerian graph} if it is quotient of a union of simple cycles. In the sequel, a directed labeled graph is said Eulerian if the graph obtained by forgetting labels and edge orientation is Eulerian.

\begin{lemma}[Euler-Hierholzer theorem]\label{Euler-Hierholzer} A connected graph  $S$ is  \emph{Eulerian} if and only if the degree of each vertex is an even number. 
\end{lemma}

We now prove the first part of Lemma \ref{Lem:ErrorTerms}. Let $\pi\in \mcal P(V)$ such that, in the $i$-th component $S$ of $T^\pi$, the edges labeled by Wigner matrices are of even multiplicity. Note that each vertex of $T^\pi_X$ has even degree. In the graph $T$, each vertex is adjacent to one edge labeled by a Wigner matrix and one edge labeled by a deterministic matrix. Hence any vertex of $S$ is adjacent to the same number of edge from one and the other family: the degree of a vertex in a deterministic component of $S$ equals its degree in $T^\pi_X$. By Euler-Hierholzer each deterministic component of $S$ is Eulerian. In particular it has no cutting edge and so $q_{i,2}'=0$, which proves the first part of the lemma. 

To prove second part of Lemma \ref{Lem:ErrorTerms}, we use the following notion.

\begin{definition}\label{Def:Pruning}
Given a connected component ${\mcal {GDC}}_i$ of $\mcal {GDC}$, we denote by $\mcal Prun(\mcal {GDC}_i)=(\widetilde {\mcal V}_i, \widetilde {\mcal E}_i)$ the undirected graph obtained by first suppressing the leaves of $\overline{\mcal {GDC}_i}$ and the edges incident to them, and then repeat this process until there is no leaf remaining.  We call  $\mcal Prun({\mcal {GDC}_i})$ the \emph{pruning} of  $ {\mcal {GDC}}_i$.
\end{definition}

\begin{figure}[t]
\leavevmode\kern-3em
\includegraphics{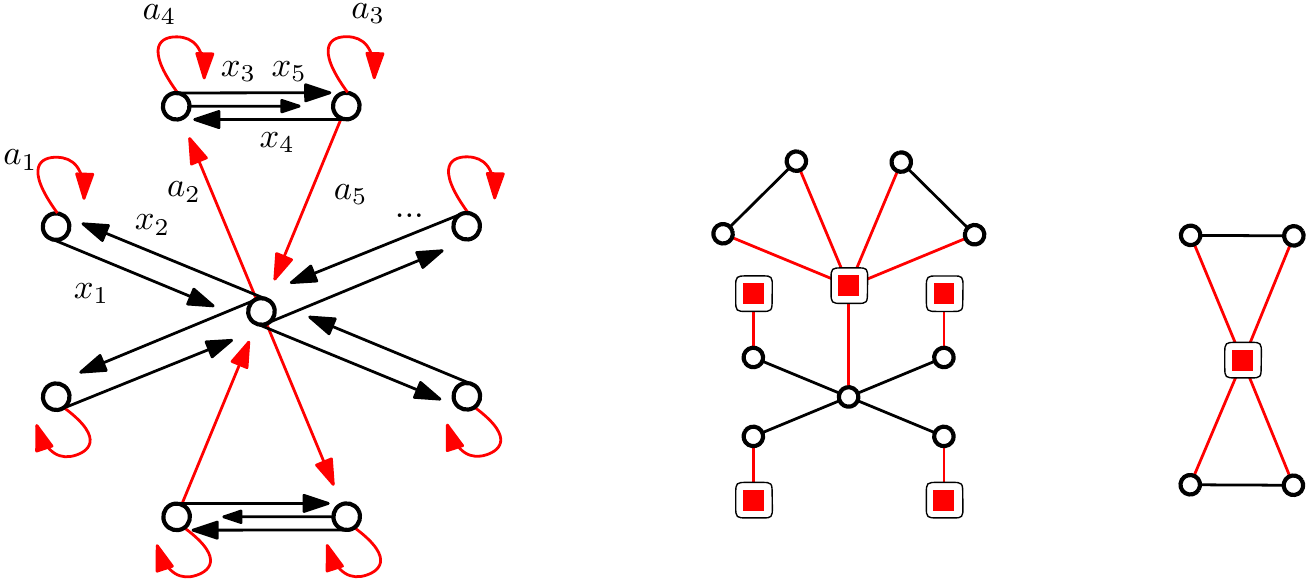}

\caption{Left: a labeled graph $T^\pi$ quotient of a cycle of length 28. Middle and right: the graph $\overline{\mcal {GDC}}$ of deterministic components with edge multiplicity forgotten and its pruning  $\mcal Prun({\mcal {GDC}})$. Note the subgraph $T^\pi_A$ (in red) has a connected component whose graph of t.e.c.c. have 4 leaves.
}
\label{Fig9:Pruning}
\end{figure}

If $\overline{\mcal {GDC}_i}$  is not a tree and its pruning $\mcal Prun({\mcal {GDC}_i})$ is a non trivial graph. Recall the notation $\mathfrak f(C)$ from Definition \ref{Def:TEC2} for the number of leaves in forest of two-edge connected components of $C$. 
\begin{lemma}\label{VariationParityArgu} Let $C$ be a deterministic component of $T^\pi$. 
\begin{enumerate}
	\item If $C$ is suppressed by the pruning process, it has no cutting edge.
	\item If $C$ is a vertex of $\mcal Prun({\mcal {GDC}_i})$, then $\mathfrak f (C)\leq \text{deg} (C)$, where $\mrm{deg}$ means the degree of the vertex in the graph $\mcal Prun( {\mcal {GDC}_i})$.
\end{enumerate}
\end{lemma}

Taking for granted Lemma \ref{VariationParityArgu} momentarily, we conclude the proof of Lemma  \ref{Lem:ErrorTerms}: assume that $\overline{\mcal {GDC}_i}$  is not a tree and let us prove that necessarily $q_{2,i}+q_{2,i}'\leq 0$.  Note that when pruning a graph we suppress one edge for each leaf, so we have
	\eq
		 q_{2,i} = \big( |\mcal V_i| - |\bar{\mcal E}_i| \big) =  \big( |\widetilde {\mcal V}_i| - |\widetilde {\mcal E}_i| \big).
	\qe
Moreover, as for $\mcal {GDC}_i$, the graph $\mcal Prun({\mcal {GDC}}_i)$ has vertices of two kinds, according to Definition \ref{Def:GCC}. We denote by $\widetilde {\mcal V}_{X,i}$ the set of vertices coming from the vertex set of $T^\pi$, and by $\widetilde {\mcal V}_{A,i}$ the set of vertices coming from the connected components of $T^\pi_A$.  Recall the classical formula in graph theory $|\widetilde {\mcal E}_i|= \sum_{v \in \widetilde {\mcal V}_i} \frac{\mrm{deg}\, v}2$, where $\mrm{deg}\, v$ is the number of neighbours of $v$ in $\mcal Prun({\mcal {GDC}_i})$.  We hence get
	\eq
		\lefteqn{q_{2,i}+q_{2,i}' =   \big( |\widetilde {\mcal V}_i| - |\widetilde {\mcal E}_i| \big) +  \Big( \frac{\mathfrak f (T^\pi_A)}2 - |C_{A}| \Big)}\\
		& = & \sum_{v\in \widetilde {\mcal V}_{X,i}} \Big( 1- \frac{\mrm{deg}\, v}2 \Big)  + \sum_{C\in \widetilde {\mcal V}_{A,i}}  \Big(1-\frac{\mrm{deg}\, C}2 \Big) + \sum_{C\in C_{A,i}} \Big(  \frac{\mathfrak f (C)}2 -1 \Big) .
	\qe
In the above formula, $C_{A,i}$ denotes the set of connected components of $T^\pi_A$ which belongs to $\mcal {GDC}_i$. The first part of Lemma \ref{VariationParityArgu} implies that the deterministic components $C$ that are erased by the pruning process satisfies $\mathfrak f(C)=2$. Hence in the right hand side of the above formula the last sum can be restricted to the sum over $C\in \widetilde {\mcal V}_{A,i}$
\eq
		q_{2,i}+q_{2,i}'= \sum_{v\in \widetilde {\mcal V}_{X,i}} \Big( 1- \frac{\mrm{deg}\, v}2 \Big) + \sum_{C\in \widetilde {\mcal V}_{A,i}} \Big(  \frac{\mathfrak f (C)}2 -\frac{\mrm{deg}\, C}2 \Big).
\qe
 Since $\mcal Prun({\mcal {GDC}}_i)$ has no leaves we have $\mrm{deg}\, v\geq 2$ for any $v\in \widetilde {\mcal V}_{X,i}$, and the second part of Lemma \ref{VariationParityArgu} states that $ \mathfrak f (C)  \leq  \mrm{deg}\, C$. This proves that $q_{2,i}+q_{2,i}'\leq 0$ and finishes the proof of Lemma \ref{Lem:ErrorTerms} - provided Lemma \ref{VariationParityArgu} holds true. $\square$

We now turn to the proof of Lemma \ref{VariationParityArgu}, using the following notions.

\begin{definition}\label{cutting vertex} Let $S=(V,E)$ be a connected graph. We say that $\omega\in V$ is a \emph{cutting vertex} of $S$ if there is a partition $V=V'\sqcup \{\omega\} \sqcup V''$, $V',V''\neq \emptyset$, such that there are no edge between a vertex of $V'$ and a vertex of $V''$. Denoting by $S'$ the subgraph of $S$ obtained by removing the vertices of $V''$ and the edges attached to it, we say that $S'$ is a \emph{factor} of $S$ with \emph{base} $\omega$. 
\end{definition}

\begin{lemma}\label{EulerArg} If $S$ is Eulerian, each factor of $S$ is Eulerian.
\end{lemma}

\begin{proof} Let $S'=(V',E')$ be a factor of $S$ with base $\omega$. By Euler-Hierholzer theorem, all vertices $v\neq \omega$ have even degree in $S$. Since $S'$ is factor, they have even degree in $S'$. Moreover, the formula $|E'|=\sum_{v\in V'}  \frac{\mrm{deg}\, v}2$ implies $  \mrm{deg}\, \omega  = 2 |E'| - \sum_{v \neq \omega}  \mrm{deg}\, v$ so $\mrm{deg}\, \omega $ is even. Hence $S'$ is Eulerian.
\end{proof}

\begin{proof}[Proof of Lemma \ref{VariationParityArgu}]
Let $S$ denote the $i$-th component of $T^\pi$. To prove the first part of the lemma we decompose the pruning process of $\overline{\mcal{GDC}_i}$, and construct a sequence of factors $S^{(1)},S^{(2)} \etc S^{(m)}=S^{(max)}$ of $S$ as follows. Note first that each suppression of a leaf in the pruning process either removes 
\begin{itemize}
	\item {(\it $A$-step)} a deterministic component $C$ and an associated edge $(v,C)$, for some $v$ of $T^\pi_X$;
	\item {(\it $X$-step)} or a vertex $v$ of $T^\pi_X$ and the edge adjacent to it.
\end{itemize}

Assume the first leaf suppression is an $A$-step that removes the vertex $C$ and the edge $(v,C)$ in $\overline{\mcal {GDC}_i}$. We then set $S^{(1)}$ the graph obtained from $S$ by removing all edges of $C$ and the vertices that remain isolated after this removal. Necessarily $v$ is a cutting vertex so $C$ and $S^{(1)}$ are factors of $S$. Lemma \ref{EulerArg} implies that $C$ is Eulerian, and so has no cutting edge. If the first leaf suppression is an $X$-step that removes the vertex $v$, we set $S^{(1)}$ the graph obtained from $S$ by removing the vertex $v$ and its adjacent edge. It is also a factor of $S^{(0)}$. 

We pursue this construction with $S^{(1)}$ instead of $S$, getting iteratively a sequence of factors $S^{(1)},S^{(2)}, \dots S^{(m)}=S^{(max)}$. This shows inductively that the deterministic components removed by the pruning process have no cutting edge.  We have proved the first part of Lemma \ref{VariationParityArgu}.

We now prove the second part of Lemma \ref{VariationParityArgu}. Note for the sequel that by Lemma \ref{EulerArg} the last factor $S^{(max)}$ of $S$ is an Eulerian graph. 
We now assume that $C$ is a deterministic component that is a vertex of $\mcal Prun({\mcal {GDC}_i})$, i.e. it has not been removed by the pruning process. It is a deterministic component of $S^{(max)}$, but it is not a factor. We show that its degree in $\mcal Prun({\mcal {GDC}_i})$ is not smaller than $\mathfrak f( C)$. Since $\mcal Prun({\mcal {GDC}_i})$ has no leaves, then $\mrm{deg} \, C\geq 2$, hence the property is obvious if $\mathfrak f(C)=2$. Assume from now that $\mathfrak f(C)\geq 3$ and, for a contradiction, that $\mrm{deg} \, C<\mathfrak f(C)$. 

Let $C'$ be a two-edge connected component of $C$, that is a leaf in the tree $\mcal T\hspace{-2pt}ec(C)$ of two-edge connected components of $C$. Since $C'$ is a leaf of $\mcal T\hspace{-2pt}ec(C)$, it is a factor of $C$ or a group of self-loops based on a vertex $\omega$ adjacent to a cutting edge. Since $\mrm{deg} \, C<\mathfrak f(C)$ we can find such a $C'$ that has no vertex adjacent to an edge labeled by a Wigner matrix in $S^{(max)}$. Thus no vertex of $C'\setminus\{\omega\}$ is adjacent to an edge of $S^{(max)}\setminus C'.$ This means that either $C'$ is a factor of $S^{(max)}$ or $C'$ is a group of self-loops.  Since $S^{(max)}$ is Eulerian, then $C'$ is Eulerian: all vertices of $C'$ have even degree in $C'$. But in $S^{(max)}$ the vertex $\omega$ is also  adjacent to a cutting edge, so its degree in $S^{(max)}$ is its degree in $C'$ increased by one: $S^{(max)}$ has a vertex of odd degree, which is in contradiction with Euler-Hierholzer theorem. This concludes the proof of Lemma \ref{VariationParityArgu}
\end{proof}

Proving Lemma \ref{VariationParityArgu} validates the proof of Lemma \ref{Lem:ErrorTerms}, and then the proof of the main result of this section, namely Proposition \ref{prop:topanalys}.

\section{Characterization of the limit}\label{Sec:CharLim}

\subsection{Asymptotic formula}
In this section, we only consider the asymptotics of $\tau^{(2)}$ and do not consider again the statistics $\tau^{(1)}$ of first order. We denote simply $\omega_X$ for the weight $\omega_X^{(2)}$ associated to Wigner matrices of Definition \ref{Def:OmegaX}.

Let $\mbf Z$ be a collection of Gaussian random variables, and assume that $\mbf Z$ is stable by complex conjugate, that is $Z\in \mbf Z$ implies $\overline{Z} \in \mbf Z$. 
We recall that $\mbf Z$ is Gaussian whenever it satisfies Wick formula: for any $n\geq 2$ and any $Z_1\etc Z_n\in \mbf Z$:
\eqa\label{Eq:Wick} \E[Z_1 \cdots Z_n] = \sum_{\sigma \in
  \mcal P_2(n)} \prod_{\{i,j\}\in \sigma } \E[Z_iZ_j] ,
\qea where $\mcal P_2(n)$ denotes the set of pairings
of $[n]$, i.e. partitions whose blocks all have size two. In above formula, the product is over all blocks $\{i,j\}$ of the pairing $\sigma$.

We prove that $\tau^{(2)}$ is asymptotically Gaussian and give an asymptotic formula for the second order distribution in the following Proposition.

\begin{proposition}\label{Prop:AsypGauss} 
For any pair $\{i,j\}$ of indices, we denote by $ \mcal P(i ,j)$ the set of partitions $\pi$ of the vertex set of $T_i\sqcup T_j$ with the following properties: 
\begin{enumerate}
	\item Either $(T_i\sqcup T_j)^\pi$ is of double unicyclic type, and both graphs $\overline{\mcal {GDC}}(T^\pi_i)$ and $\overline{\mcal {GDC}}(T^\pi_j)$ are unicyclic. Hence the two latter cycles correspond to the same number $K\geq 1$ of edges associated to Wigner matrices (possibly intertwined with edges associated to deterministic matrices) and the cycle of $\overline{\mcal {GDC}}\big((T_i\sqcup T_j)^\pi\big)$ is then obtained by twining pair-wise these $K$ edges of these cycles.
	\item Either $(T_i\sqcup T_j)^\pi$ is of 2-4 tree type, and both graphs $T^\pi_i$ and $ T^\pi_j$ are of double tree type. A pair $\bar e_i$ of twin edges of $T^\pi_{X,i}$ and a pair $\bar e_j$ of twin edges of $T^\pi_{X,j}$ are then twined to form a group of edges of multiplicity 4 in $(T_i\sqcup T_j)^\pi$.
	\item In each case, twin edges associated to Wigner matrices in $(T_i\sqcup T_j)^\pi$ are associated to a same Wigner matrix.
\end{enumerate}

Then we have
	$$\tau^{(2)} =\sum_{\sigma \in
  \mcal P_2(n)} \prod_{\{i,j\}\in \sigma } \tau_{con}^{(2)}(i,j)+o(1),$$
 with
 	\eqa\label{Eq:TauCon}
		 \tau_{con}^{(2)}(i,j) =  \sum_{ \substack{\pi \in \mcal P(i ,j)}}\omega_X(\pi) \omega_A(\pi).
	\qea
The contributions $\omega_X=\omega_X^{(2)}$ and $\omega_A(\pi)$ are computed by taking $T = T_i\sqcup T_j$ in Definitions \ref{Def:OmegaX} and \ref{Def:TEC}.
\end{proposition}

The notation $\tau_{con}^{(2)}$ serves to emphasize that the quotient graph is connected. 
The rest of the subsection is devoted to the proof of Proposition \ref{Prop:AsypGauss}. Let $\pi$ be a valid partition. 
We denote by $\sigma =\sigma(\pi)$ the partition of $[n]$ such that $B=\{\ell_1\etc \ell_p\}$ is a block of $\sigma$ whenever there is a connected component of $T^\pi$ formed by these graphs $T^\pi_{\ell_1} \etc T^\pi_{\ell_p}$. By the $X$-connectedness criterion of Lemma \ref{Lem:Key}, if $\omega_X(\pi)\neq 0$ then $\sigma$ has no singletons.

By independence of the Wigner matrices and of the entries, we have $\omega_X(\pi) =  \prod_{ B \in \sigma}\omega_X(\pi_B)$, where $\omega_X(\pi_B)$ is defined by taking $T =\sqcup_{j\in B} T_j$ in Definition \ref{Def:OmegaX}. The similar property for $\omega_A(\pi)$ up to a negligible error term follows from the next lemma.

\begin{lemma}\label{Lem:SeparateCC} Let $\pi \in \mcal P(V)$ be a partition such that the components of $T^\pi_A$ are two-edge connected. Recalling that $C_A=C_A(\pi)$ is the set of connected components of $T^\pi_A$, we then have
		$$\frac 1 {N^{| C_A|}}\Tr^0 [T^\pi_A] = \prod_{ C \in C_A} \frac 1 N \Tr^0[C]+o(1).$$ 
\end{lemma}

\begin{proof} We write $\pi\leq \pi'$ to mean that the blocks of $\pi$ are included in the blocks of $\pi'$, and so $T^{\pi'}$ is a quotient of $T^\pi$. With notations as in the lemma, we have
	\eq
		\frac 1 {N^{| C_A|}}\Tr^0 [T^\pi_A] = \frac 1 {N^{| C_A|}} \sum_{\substack{\pi \leq \pi'  \in \mcal P(V) \\ \mrm{injective} }} r(T_A, \pi')
	\qe
and
	\eq
		\prod_{ C \in C_A} \frac 1 N \Tr^0[C] = \frac 1 {N^{| C_A|}} \prod_{ C \in C_A} \sum_{\substack{\pi'   \in \mcal P(V_C) \\ \mrm{injective} }} r(C_A, \pi')
	\qe
Hence
	$$
	\frac 1 {N^{| C_A|}}\Tr^0 [T^\pi_A] - \prod_{ C \in C_A} \frac 1 N \Tr^0[C] =- \frac 1 {N^{| C_A|}} \sum_{\pi' } r(T_A, \pi')$$
where the sum is over all $\pi\leq \pi' \in \mcal P(V)$, whose restriction on each connected component of $T^\pi_A$ is injective, but which is not injective. Hence such a choice reduces the number of connected components. Since the graphs are two-edge connected, with the bound of Mingo and Speicher from \cite{MS11} we deduce that  $\sum_{\pi' } r(T_A, \pi') = O(N^{| C_A|-1})$.
\end{proof}

From \eqref{Eq:ConclusionPartIbis}, by additivity of the topological parameters and by asymptotic multiplicativity of the $\omega$-functions we can then deduce an asymptotic factorization with respect to connected components
	\eqa\label{Eq:ProofCovCumul}
		\tau^{(2)} =\sum_{\sigma \in
  \mcal P(n)} \prod_{ B\in \sigma } \sum_{\substack{ \pi \in \mcal P(\sqcup_{j\in B}V_j) \\ \mrm{valid, \ connected}}} \omega_X(\pi) \omega_A(\pi)+o(1),
 	\qea
where the $\omega$-functions are defined by taking $T =\sqcup_{j\in B} T_j$ in Definitions \ref{Def:OmegaX} and \ref{Def:TEC}, and we can assume $|B|\geq 2$ for all blocks $B$ of $\sigma$. 

\begin{lemma}\label{Lem:PetiteArete}Let $\pi$ be a valid partition such that $\omega_X(\pi)\neq 0$ and assume that  the $i$-th connected component of $T^\pi$ of double cycle type. Then, there are two different cycles $T_j, T_{j'}$, $j\neq j' \in [n]$, such that each a group of twin edges of $T^\pi_X$ in the cycle of $\overline{\mcal {GDC}_i}$ consists of an edge from $T_j$ and an edge from $T_{j'}$.
\end{lemma}

\begin{proof} Denote by $E_0$ the set of groups of twin edges labeled by Wigner matrices in the cycle of $\overline{\mcal {GDC}_i}$. 
Assume that a cycle $T_j$ has exactly one edge labeled by a Wigner matrix in a group $\bar e \in E_0$ of adjacent edges. For a contradiction, assume moreover that there is another group edges $\bar e'\in E_0$ in the cycle that comes from others cycle $T_{j'}, T_{j''}$, $j'\neq j\neq j''$. We denote by $T^\pi\setminus{\bar e'}$ the graph obtained from $T^\pi$ by removing the two edges of $\bar e'$. Since the removal does disconnect the graph nor change the parity of vertices, $T^\pi\setminus{\bar e'}$ is an Eulerian graph. The removal of the edges of $\bar e$ disconnects $T^\pi\setminus{\bar e'}$. The fact that a single edge of $T_j$ belongs to $\bar e$ is hence in contradiction with the parity argument of Lemma \ref{Lem:ParityArgument}. This implies that all $T_j$ has one edge in element of $E_0$, and so another graph $T_{j'}$ has the same property.

To finish the proof of lemma, we now assume that each group of edges in $E_0$ comes from a single cycle. Since all group of edges labeled by Wigner matrices out of the cycle are of multiplicty 2, Lemma \ref{Lem:ParityArgument} implies that each group of twin edges of the $i$-th component of $T^\pi$ labeled by Wigner matrices come from a cycle. The $X$-connectedness criterion (Lemma \ref{Lem:Key}) hence implies $\omega_X(\pi)=0$. \end{proof}

We can now finish the proof of Proposition \ref{Prop:AsypGauss}.
Assume for $\pi$ valid that a component $S$ of $T^\pi$ is made of at least 3 graphs among $T^\pi_1\etc T^\pi_n$ and let us prove that $\omega_X(\pi)=0$.
\begin{itemize}
	\item Assume $S$ is of 2-4 tree type. By the parity argument, only the edge labeled by Wigner matrices of multiplicity 4 can come from two graphs, so at least one graph is $X$-disconnected from all other graphs and  Lemma \ref{Lem:Key} implies $\omega_X(\pi)=0$. 
	\item Assume $S$ is of double unicyclic type. By Lemma \ref{Lem:PetiteArete}, only the edge labeled by Wigner matrices that belongs to the groups in the cycle of $\overline{\mcal {GDC}}(S)$ can come from two graphs, and the same conclusion holds.
\end{itemize}

Hence in Formula \eqref{Eq:ProofCovCumul} we can restrict the sum to pairings $\sigma\in \mcal P_2(n)$. That the only pairings $\sigma$ for which $\omega_X(\pi)\neq 0$ are those such that the blocks $\{i,j\}$ of $\sigma$ are in $\mcal P(i,j)$ is a consequence of same arguments of $X$-connectedness as before. Finally, by independence of the matrices and the centeredness in the definition of $\omega_X$, we have $\omega_X(\pi)=0$ when there exist twin edges associated to different Wigner matrices. This finishes  the proof of Proposition \ref{Prop:AsypGauss}.

\subsection{Computation of the covariance: even moments and vanishing pseudo-variance}\label{EvenMomNullPV}

By Proposition \ref{Prop:AsypGauss}, in the sequel we can assume $n=2$, so $T=T_1\sqcup T_2$ is the union of 2 simple cycles. We shall compute the limit of $\tau^{(2)}_{con}(1,2)$ defined in Equation \eqref{Eq:TauCon} and prove the covariance formulas of Theorems \ref{MainTh2} and \ref{MainTh4}. We use the shorthand $\tau^{(2)}_{con}:=\tau^{(2)}_{con}(1,2)$.

By Proposition \ref{Prop:AsypGauss}  and  Lemma \ref{Lem:SeparateCC}, we have
	\eqa\label{MainAsymptoticFormula}
		\tau_{con}^{(2)}=  \sum_{ \pi\in \mcal P(1,2)}\omega_X(\pi) \prod_{ C \in C_A} \frac 1 N \Tr^0[C]+o(1),
	\qea
where $\pi\in \mcal P(1,2)$ means that $T^\pi$ is of double unicyclic or 2-4 tree type graph as in Definition \ref{Def:DUG}, and $C_A=C_A(\pi)$ denotes the set of connected components of $T^\pi_A$. We denote by $2m$ and $2n$ the lengths of the cycles $T_1$ and $T_2$ respectively (they are even since the letters $X_i$ and $A_i$ alternated in Definition  22 of the matrices $M_j$). By a simple parity argument for the number of double edges coming from $T_1$ and $T_2$, if $m$ and $n$ do not have the same parity, there is no $\pi\in \mcal P(V^\pi)$ such that $T^\pi$ is of double unicyclic or 2-4 tree type, and so $\tau_{con}^{(2)}\limN 0$. 

In this section, we assume that $m$ and $n$ are even. Assume that $T^\pi$ is of double unicyclic type graph. We claim that the double cycle of $\mcal {GDC}(T^\pi)$ has an even number $2\ell$ of double edges labeled by Wigner matrices, for $\ell\geq 1$.  Indeed, each $T_i$ has an even total number of edges. It has also an even number of edges out of the cycle and there is exactly one edge of each $T_i$ for each double edge of the cycle. We use in the sequel the idea that a 2-4 tree type graph is a degenerated version of a double unicyclic type graph with $2\ell=2$. This consideration is important latter since the expression \eqref{MainAsymptoticFormula} depends on the injective traces $ \Tr^0[C]$ whereas we want an expression in terms of the parameters of the deterministic matrices, i.e. normalized traces of entry-wise products. 

\subsubsection{Computation of weights $\omega_X$}

\begin{definition} \begin{enumerate}
	\item We say that two twin edges $e_1$ and $e_2$ are opposite if the source of $e_1$ is the target of $e_2$ and reciprocally, and that they are parallel if they have same source and same target. 
	\item Let $\pi \in \mcal P(1,2) $ such that  $T^\pi $ is of double unicyclic type. We say that $T^\pi$ is of opposite type if all twin edges of $T^\pi_X$ are opposite, and of parallel type otherwise.
	\end{enumerate}
\end{definition}

For a parallel type graph, twin edges labeled by a Wigner matrix on the double cycle are all parallel, twin edges outside the double cycle are always opposite.

We compute easily the value of $\omega_X(\pi)$ from its definition in Eq. \eqref{Eq:OmegaX}. 
\begin{itemize}
	\item Let $\mcal {DU}^{opp}$ denote the set of partitions $\pi\in \mcal P(1,2)$ such that $T^\pi$ is of opposite double unicyclic type. Since all entries of Wigner matrices have normalized variance, for $\pi \in \mcal {DU}^{opp}$ we have
	\eqa\label{Eq:W1}
		\omega_X(\pi)=1.
	\qea
	\item Let $\mcal {FT}$ denote the set of $\pi\in \mcal P(1,2)$ such that $T^\pi$ is of 2-4 tree type. Then $\pi \in \mcal {FT}$ implies
	\eqa\label{Eq:W2}
		\omega_X(\pi)=\E[|x^{(\pi)}_{12}|^4]-1,
	\qea
 where $X^{(\pi)}=\big( \frac{x^{(\pi)}_{ij}}{\sqrt N}\big)_{ij}$ stands for the Wigner matrix associated to the edge of multiplicity 4 in $T^\pi_X$. 
 	\item Let $\mcal {DU}^{par}$ the set of partitions $\pi\in\mcal P(1,2)$ such that  $T^\pi$ is of parallel double cyclic type. For all $\pi \in \mcal {DU}^{par}$ we have
 	\eqa\label{Eq:W3}
		\omega_X(\pi)=\theta(\pi):=\prod_{k=1}^K\E[{x^{(i_k)}_{12}}^2],
	\qea
 where $X^{(i_k)}=\big( \frac{x^{(i_k)}_{ij}}{\sqrt N}\big)_{ij}$, $k=1\etc K$, stands for the Wigner matrices associated to the double cycle.
\end{itemize}
 
From now and until the rest of Section \ref{EvenMomNullPV}, we also assume that the Wigner matrices have null pseudo-variance. 
Hence $\theta(\pi)$ in \eqref{Eq:W3} vanishes and we get from \eqref{MainAsymptoticFormula}	
\eq
		\tau^{(2)}_{con} &=& 
		 \sum_{ \pi \in \mcal {DU}^{opp}} \bigg( \prod_{ C \in C_A(\pi)} \frac 1 N \Tr^0[C]\bigg)\\
		 && +  \sum_{ \pi \in \mcal {FT}} \bigg(\big( \E[|x^{(\pi)}_{12}|^4]-1 \big) \prod_{ C \in C_A(\pi)} \frac 1 N \Tr^0[C] \bigg)+o(1).
\qe

The aim is now to see how appear the parameters of the deterministic components. For that given a partition $\pi$ we associate a graph $G^\pi_X$ whose edges are labeled by the Wigner matrices and a graph $G^\pi_A$ labeled by the deterministic matrices. Recall that $T^\pi_X$ (resp. $T^{\pi}_A$) is the subgraph of $T^{\pi}$ whose edges are labeled by the Wigner (resp. deterministic) matrices.

\begin{definition}\label{SkullCompl} Let $\pi \in \mcal P(1,2)$. We denote $G^\pi_X$ the graph labeled by Wigner matrices obtained from $T^\pi$ by contracting the edges labeled by deterministic matrices to vertex: we identify the source and the target for each of those edges and we remove them. We denote $G^\pi_A$  the graph labeled by deterministic matrices that is the smallest quotient $T^{\pi_0}_A$ such that $G^\pi_X = G^{\pi_0}_X$. 
\end{definition}

Let $\mcal G=\mcal G(n,m)$ denote the set of all possible graphs $G^\pi_X$ for $\pi\in \mcal P(i,j)$. We also set
\eq
	\mcal G^{opp} & = & \{ G^\pi_X \in \mcal G \, | \, \pi \in \mcal {DU}^{opp}\}.
\qe

\subsubsection{Opposite type, first case}
 For any $\ell\geq 1$, we denote by $ \mcal {DU}^{opp}_{\ell}$ (resp. $\mcal G^{opp}_\ell)$ the set of partitions $\pi \in \mcal {DU}^{opp}$ (of graphs $G^\pi_X$) such that there are $\ell$ double edges labeled by Wigner matrices in the double cycle.

\begin{figure}[t]
\includegraphics{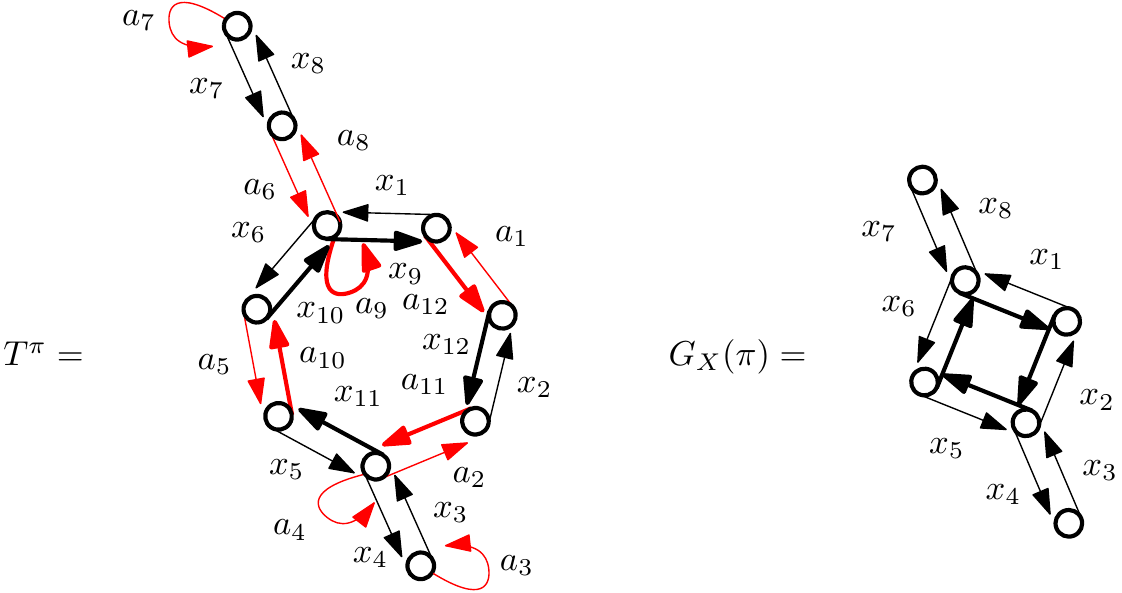}

\caption{Left: a labeled graph $T^\pi$ such that $\pi \in   \mcal {DU}^{opp}_{4}$. Right: the graph $G^\pi_X$. }
\label{fig_14_OppGrand.pdf}
\end{figure}

\begin{figure}[t]
\includegraphics{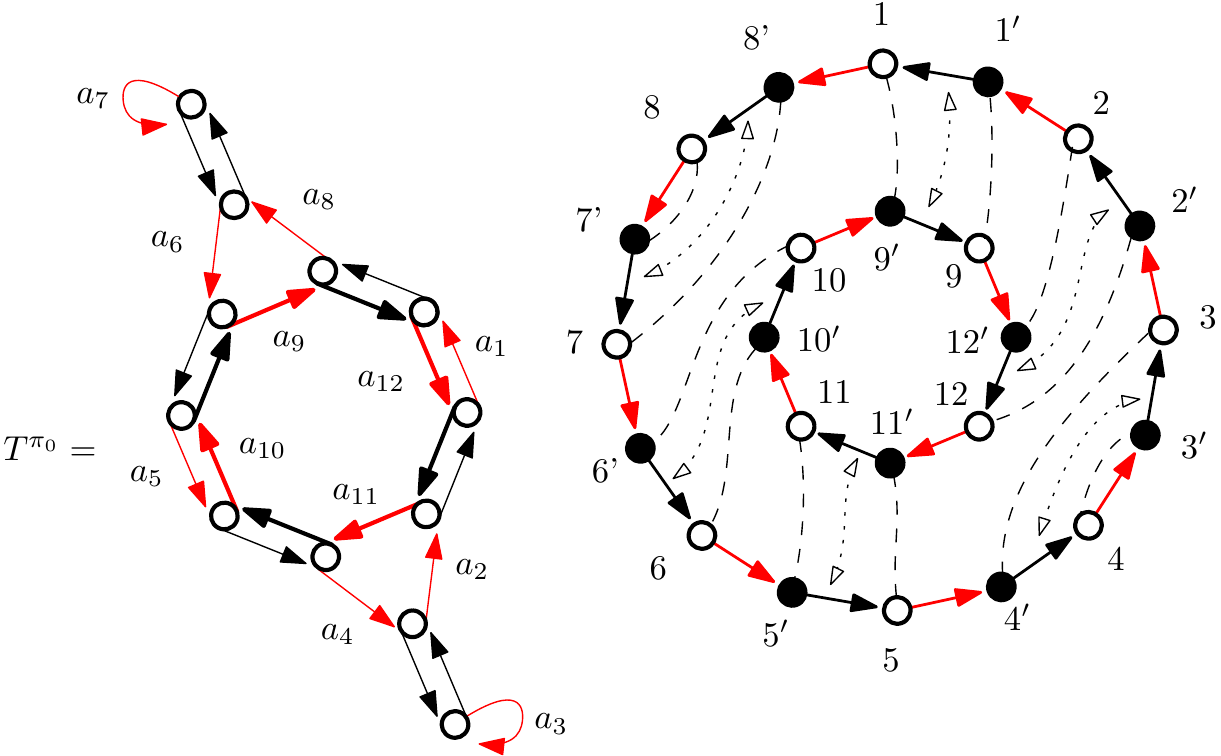}

\caption{Left: The graph $T^{\pi_0}$ for the minimal partition $\pi_0$ such that $G_X(\pi_0) = G^\pi_X$ where $\pi$ is as in Figure \ref{fig_14_OppGrand.pdf}. Note that  $G^\pi_A$ consists of the union of the components of $T^{\pi_0}$ labeled by deterministic matrices. Right: the annulus $\mcal Ann^{opp}$, with dashed lines to represent to identifications of vertices made by $\pi_0$, and dotted lines with arrows to represent the partition $\sigma_{G^\pi_X}$.}
\label{fig_15_OppGrand2.pdf}
\end{figure}

Let $G \in \mcal G^{opp}_{2\ell}$ for $\ell\geq 2$. All partitions $\pi$ such that $G^\pi_X=G$ have same graph $K(G):=G^\pi_A$ that we call the \emph{complement} of $G$. The connected components of this graph $K(G)$ are simple oriented cycles $C_1\etc C_k$. Each partition $\pi\in \mcal P(1,2)$ such that $G^\pi_A = K(G)$ corresponds to a $k$-tuple $(\pi_1\etc \pi_k)$ where $\pi_i$ is a partition of $C_i$. The relation between trace and injective trace of graphs implies (with $V_{C_i}$ the vertex set of the $i$-th cycle $C_i$ of $K(G)$)
	\eq
		\sum_{ \substack{\pi \in  \mcal {DU}^{opp}_{2\ell} \\ \mrm{s.t. \ } G^\pi_X=G} } \prod_{ C \in C_A(\pi)} \frac 1 N \Tr^0[C]  &=& \prod_{i=1}^k \sum_{\pi_i\in \mcal P(V_{C_i})}   \frac 1 N \Tr^0[C_i^\pi]
		=  \prod_{i=1}^k \frac 1 N \Tr[C_i].
	\qe

We now associate to any graph $G \in \mcal G^{opp}_{2\ell}$ a pairing $\sigma_G \in NC_2^{(2\ell)}(m,n)$, where we recall that $m$ and $n$ are the lengths of $T_1$ and $T_2$ respectively. 

\begin{definition}\label{NC1} We denote by $\mcal Ann^{opp}$ the labeled graph $T$ embedded in $\mbb C$ that consists of the annulus formed by the outer cycle $T_1$ in anticlockwise orientation, and the inner cycle $T_2$ in clockwise orientation. For any $G\in  \mcal {G}^{opp}$, we set $\sigma_G$ the pairing of the edge set of $\mcal Ann^{opp}$ labeled by Wigner matrices such that two edges belong to a same block if and only if they are twined in $G$.
\end{definition}

By contracting the $A$-edges of $\mcal Ann^{opp}$, we see $\sigma_G$ as an annulus partition of $(m,n)$ elements represented by the edges labeled $X_1\etc X_{m+n}$. Recall that $\sigma \in NC_2^{(2\ell)}(m,n)$ is non-mixing if twin edges are always labeled by a same Wigner matrix.

\begin{lemma}
The pairing  $\sigma_G$ is non crossing, and the function $G \mapsto \sigma_G$ defines a bijection from $\mcal G^{opp}_{2\ell}$ to the set of non-mixing elements of $NC_2^{(2\ell)}(m,n)$.
\end{lemma}

\begin{proof} Let $G=G_X(\pi)\in \mcal G^{opp}_{2\ell}$. If each edge labeled by a Wigner matrix in $T_{1}$ is twinned in $T^\pi$ with an edge of $T_{2}$ (so $\overline{\mcal{GDC}}$ is a cycle), then $\sigma_G$ is a \emph{spoke diagram}: necessarily $n=m$ and there is an index $j_0$ such that the $j$-th edge of $T_{1,X}$ is twinned in $T^\pi$ with the $j+j_0$-th edge of $T_{2,X}$, for any $j$ with notations modulo $n$. Hence $\sigma_G$ is a non crossing pairing.

Otherwise, they are edges labeled by Wigner matrices in $T_{1}$ or in $T_{2}$ that are twinned together, and so $T^\pi$ has double tree type factor. Recall that an annular partition of  is non-crossing if and only if either it is a spoke diagram, or it has a block of the form $\{i,i+1\etc i+j\}$ in $[n]$ or $[m]$, and the removal of this blocks yields another non-crossing partition \cite{NS06}, Remark 9.2. The pruning process shows that this nesting property is satisfied by $\sigma_G$.

Hence every $G\in G^{opp}_{2\ell}$ yields an element of $NC_2^{(2\ell)}(m,n)$. It is non mixing by the third condition in the definition of $\mcal P(i,j)$ (Proposition \ref{Prop:AsypGauss}). Reciprocally, if $\sigma$ is a pairing in $NC_2^{(2\ell)}(m,n)$, it gives a way to identify pairwise the edges of $\mcal Ann^{opp}$, and if $\sigma = \sigma_G$ then contracting the $A$-edges yields correctly the graph $G$. 
\end{proof}

Let $G=G^\pi_X \in  \mcal G^{opp}_{2\ell}$ and  $K(G) = G^\pi_A$ its complement introduced above. We consider the partition of the edge set of $\mcal Ann^{opp}$ labeled by deterministic matrices such that two edges belong to a same block if and only if they belong to a same cycle of $G^\pi_A$. Then this partition is actually the Kreweras complement $K(\sigma_G)$ of $\sigma_G$. Moreover, denoting by $C_1\etc C_k$ the cycles of $K(G) =G^\pi_A$, then
	$$ \prod_{i=1}^k \frac 1 N \Tr[C_i] \limN \phi_{ K(\sigma_G)}(a_1 \etc a_{m+n}),$$
where $\phi_{ K(\sigma)}$ is as in the statement of Theorem \ref{MainTh2}.  
We hence have proved for $\ell\geq 2$
	\eq
		\sum_{ G \in  \mcal G^{opp}_{2\ell}} \sum_{ \substack{\pi \in  \mcal {DU}^{opp}_{2\ell} \\ \mrm{s.t. \ } G^\pi_X=G} } \prod_{ C \in C_A} \frac 1 N \Tr^0[C] \limN \sum_{\substack{ \sigma \in  NC_{2}^{(2\ell)}(m,n)
\\  \mrm{non-mixing}
 }} \phi_{ K(\sigma)}(a_1 \etc a_{m+n}).
	\qe

\subsubsection{4-2 type}
We set $\mcal G^{FT}$ the set of graphs $G$ of the form $G=G^\pi_X$ where $\pi \in \mcal {FT}$, i.e. $T^\pi$ is of 4-2 tree type, and fix $G \in \mcal G^{FT}$.

\begin{figure}[t]
\includegraphics{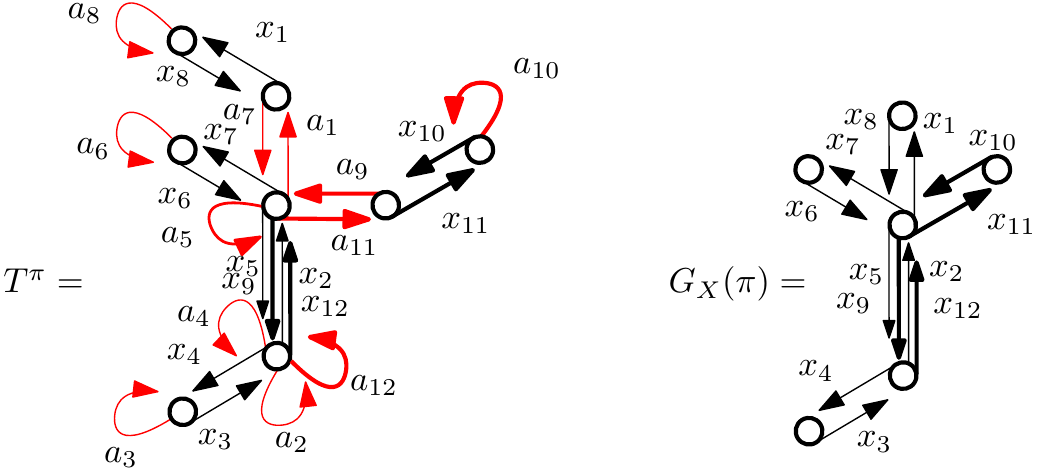}

\caption{Left: a labeled graph $T^\pi$ such that $\pi \in   \mcal T^{42}$. Right: the graph $G^\pi_X$. }
\label{fig_16_42.pdf}
\end{figure}

\begin{figure}[t]
\includegraphics{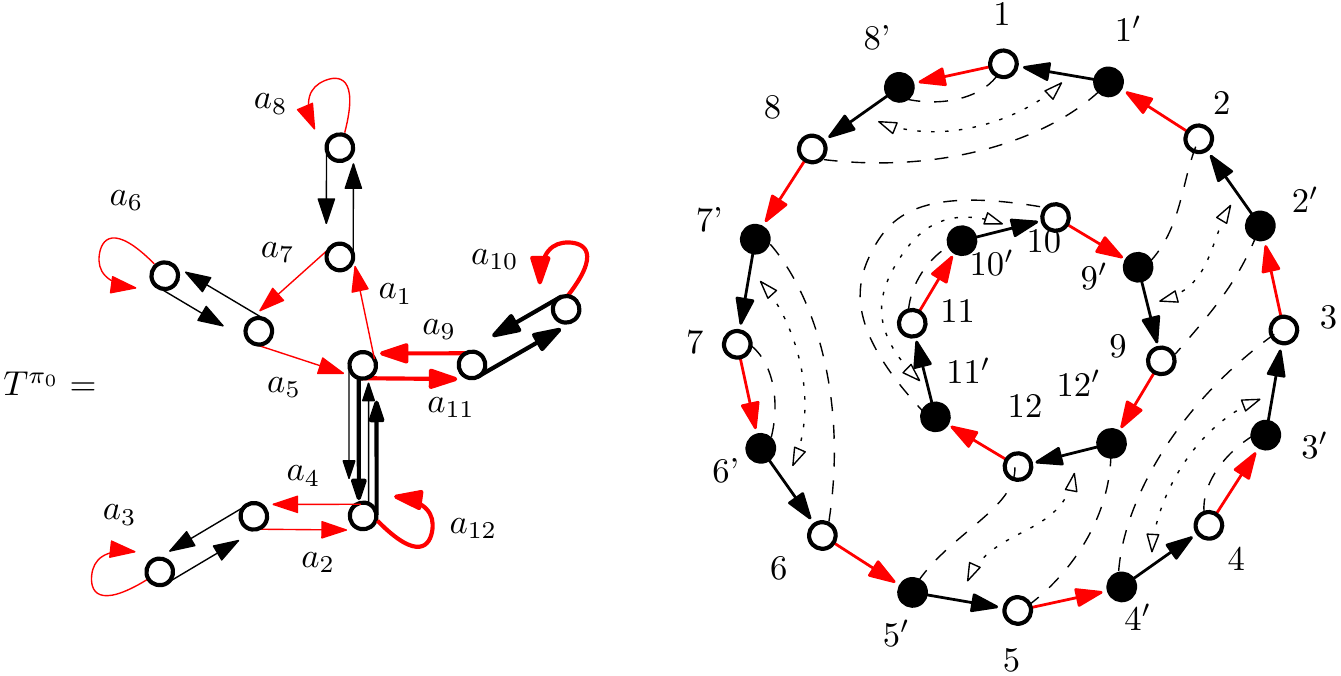}

\caption{Left: The graph $T^{\pi_0}$ for the minimal partition $\pi_0$ such that $G_X(\pi_0) = G^\pi_X$ where $\pi$ is as in Figure \ref{fig_16_42.pdf}. Edges of $T_1$ (resp. $T_2$) are the thin (resp. thick) ones. Right: the annulus $\mcal Ann^{opp}$, with dashed lines to represent to identifications of vertices made by $\pi_0$, and dotted lines with arrows to represent the pairing $\sigma_{G^\pi_X}$.}
\label{fig_17_422.pdf}
\end{figure}

For any $\pi$ such that $G^\pi_X = G$, the graph $K(G) = G^\pi_A$ depends only on $G$. Since $G$ is a fat tree (forgetting the multiplicity of edges yields a tree), it has one edge of multiplicity 4 and the other edges are of multiplicity two, the graph $K(G)$ consists in a union of cycles $C_1\etc C_k$ and of two subgraphs $C_{k+1}$ and $C_{k+2}$ (adjacent to the edge of multiplicity 4 in the minimal graph $T^{\pi_0}$). Each of these subgraphs $C_{k+1}$ and $C_{k+2}$ consists of two simple directed cycles identified by one vertex. The partitions $\pi'$ such that $G_A^{\pi'}=G^\pi_A$ are in correspondence with the $k$-tuples $(\pi_1\etc \pi_{k+2})$ where $\pi_i$ is a partition of the vertex set of $C_i$. Hence we get
	\eq
		\sum_{ \substack{\pi \in  \mcal {FT} \\ \mrm{s.t. \ } G^\pi_X=G} }  \prod_{ C \in C_A} \frac 1 N \Tr^0[C] =  \prod_{i=1}^{k+2} \frac 1 N \Tr[C_i].
	\qe
We  associate to  $G \in \mcal G^{FT}$ a partition $\sigma_G \in NC_2^{(2)}(m,n)$.

\begin{definition}\label{NC2} With $\mcal Ann^{opp}$ as in Definition \ref{NC1}, for any $G\in  \mcal G^{FT}$, we set $\sigma_G$ the pair pairing of the edge set of $\mcal Ann^{opp}$ labeled by Wigner matrices such that two edges $e,e'$ belong to a same block if and only if 
\begin{itemize}
	\item they are twins in $G$,
	\item and moreover, if $e$ and $e'$ belong to the group of edge of multiplicity 4, then $e$ and $e'$ belong to the different cycles $T_1$ and $T_2$ and have opposite orientation.
\end{itemize}
\end{definition}

The same arguments as before show that $\sigma_G$ for $G \in \mcal G^{FT}$  is a non-mixing non-crossing annulus pairing with two through strings, and the map $G \mapsto \sigma_G$ is a bijection. The two ways to connect the through strings illustrated in Figure 2 represent the two ways to form a 4-2 tree by identifying a double edge from two double trees. Denoting as before by $C_1\etc C_{k+2}$ the connected components of $K(G)$, we have
	$$ \prod_{i=1}^{k+2} \frac 1 N \Tr[C_i] \limN \tilde \phi_{ K(\sigma_G)}(a_1 \etc a_{m+n}),$$
where $\tilde \phi_{ K(\sigma)}$ is as in the statement of Theorem \ref{MainTh2}. The Hadamard products in the definition of $\tilde \phi$ result from the components $C_{k+1}$ and $C_{k+2}$, point (2) from Proposition \ref{Prop:AsypGauss} (see also Figure 1). 
We hence get
\eq
		\lefteqn{ \sum_{  \pi \in \mcal {FT}}\big( \E[|x^{(\pi)}_{12}|^4]-1\big)\prod_{ C \in C_A} \frac 1 N \Tr^0[C]}\\
		& \limN& \sum_{\substack{ \sigma NC_{ 2 }^{(2)}(m,n) \\ \mrm{non-mixing}}} \big( \E[|x^{(\sigma)}_{12}|^4]-1\big) \tilde \phi_{ K(\sigma_G)}(a_1 \etc a_{m+n}),
\qe
where non-mixing means also that the 2 through strings are labeled by the same Wigner matrix  $X^{(\sigma)}$. We emphasize that $\big( \E[|x^{(\sigma)}_{12}|^4]-1\big) $ is not the weight expected in Theorem \ref{MainTh2}.

\subsubsection{Opposite type, second case} 

Let $G \in \mcal G^{opp}_2$, namely $G=G^\pi_X$ for $T^\pi$ of double unicyclic type such that $\mcal {GDC}(T^\pi)$ has exactly 2 double edges labeled by a Wigner matrix on its double cycle. The graph $G$ is degenerated: as we contract the $A$-edges, the double cycle in $\mcal {GDC}(T^\pi)$ becomes an edge of multiplicity 4 in $G$. Hence $G$ is a 4-2 tree, namely a fat tree whose edges are all of multiplicity 2 but one is of multiplicity 4.

\begin{figure}[t]
\includegraphics{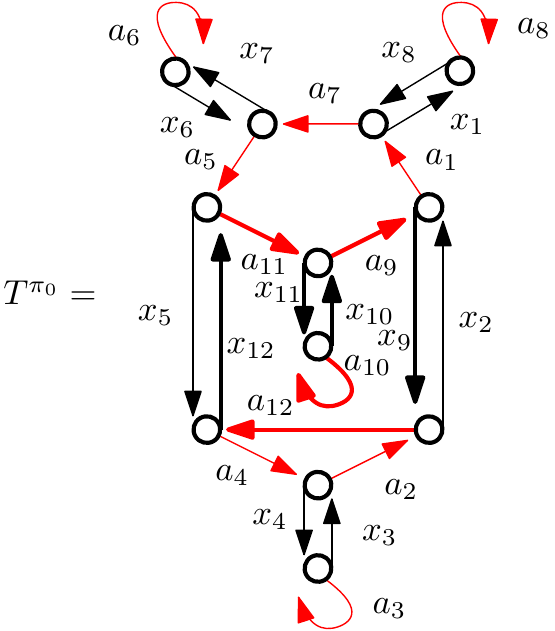}

\caption{A labeled graph $T^\pi$ such that $\pi \in   \mcal {DU}^{opp}_2$. The graph $G^\pi_X$ is the same as the rightmost picture in Figure \ref{fig_16_42.pdf}.} 
\label{fig_18_OppPetit.pdf}
\end{figure}

Nevertheless,  as for $G \in \mcal G^{opp}_{2\ell}$ for $\ell\geq 2$, it remains true that $G^\pi_A$ is a union of simple cycles $C_1\etc C_k$. In this case yet there are partitions $\pi'$ that factorize as disjoint partitions $\pi'_1\etc \pi'_k$ of the different cycles $C_i$'s, but for which $G_A^{\pi'} \neq G^\pi_A$: they are the partitions $\pi'$ such that the two double edges of the double cycle are identified to form a group of edges of multiplicity 4. Hence they are the partitions $\pi' \in \mcal {FT}$ such that $G=G_X^{\pi'}$. The graph $G_A^{\pi'}$ is a union of $k-2$ cycles $\tilde C_1\etc \tilde C_{k-2}$, where two graphs $\tilde C_i$ are obtained by identifying a vertex of cycle $C_p$ and a vertex of a cycle $C_q$. 

We hence have
	\eq
				\lefteqn{\sum_{ \substack{\pi \in  \mcal {DU}^{opp}_{2} \\ \mrm{s.t. \ } G^\pi_X=G} } \prod_{ C \in C_A(\pi)} \frac 1 N \Tr^0[C]  }\\
				&=& \prod_{i=1}^k \sum_{\pi_i\in \mcal P(V_{C_i})}   \frac 1 N \Tr^0[C_i^\pi] - \sum_{ \substack{\pi \in  \mcal {FT}\\ \mrm{s.t. \ } G^\pi_X=G} } \prod_{ C \in C_A(\pi)} \frac 1 N \Tr^0[C] \\
		&=&  \prod_{i=1}^k \frac 1 N \Tr[C_i]-\prod_{i=1}^{k-2} \frac 1 N \Tr[\tilde C_i]\\
		& \limN & \phi_{ K(\sigma)}(a_1 \etc a_{m+n}) - \tilde \phi_{ K(\sigma)}(a_1 \etc a_{m+n}).
	\qe
As a conclusion, we have proved for $m$ and $n$ even and assuming null pseudo-variance for Wigner matrices, by \eqref{MainAsymptoticFormula}
	\eq
		\tau^{(2)}_{con} & \limN & \sum_{\substack{\sigma \in   NC_{2}(m,n)\\  \text{non-mixing} 
}} \phi(a_1\etc a_{m+n}) \\
		& & \qquad+ \sum_{\substack{\sigma \in NC_{2}^{(2)}(m,n)\\
 {non-mixing} 
}}  \big( \E[|x^{(\sigma)}_{12}|^4]-2\big)  \tilde \phi_{ K(\sigma)}(a_1 \etc a_{m+n}).
	\qe
This is now indeed the expected formula \eqref{eq:thm2} in Theorem \ref{MainTh2}.

\subsection{Computation of the covariance: even moments,  general pseudo-variance}

We now have by \eqref{MainAsymptoticFormula}, with $C_A=C_A(\pi)$,
	\eq
		\tau^{(2)}_{con} &=& \sum_{ \pi \in \mcal {DU}^{opp}} \prod_{ C \in C_A} \frac 1 N \Tr^0[C] +  \sum_{ \pi \in \mcal {FT}}\big( \E[|x^{(\pi)}_{12}|^4]-1 \big) \prod_{ C \in C_A} \frac 1 N \Tr^0[C]\\
		& &  +  \sum_{ \pi \in \mcal {DU}^{par}} \theta(\pi) \prod_{ C \in C_A} \frac 1 N \Tr^0[C]+o(1),
	\qe
where $\theta(\pi)$ is defined in \eqref{Eq:W3}. As before, for any $\ell\geq 1$, we denote by $ \mcal {DU}^{par}_{2\ell}$ the set of partitions $\pi \in \mcal {DU}^{par}$ with $2\ell$ double edges by Wigner matrices in the double cycle, and by $\mcal G^{par}_{2\ell}$ the set of graphs $G\in \mcal G$ of the form $G=G^\pi_X$ for some $\pi \in  \mcal {DU}^{par}_{2\ell}$. 

Assume $G \in \mcal G^{par}_{2\ell}$ for $\ell\geq 2$. As for elements of $\mcal G^{opp}_{2\ell}$, all partitions $\pi$ such that $G^\pi_X=G$ have same graph $K(G):=G^\pi_A$, which consists of a union of simple oriented cycles $C_1\etc C_k$, and such partitions $\pi$ are in correspondence with the $k$-tuples of partitions $(\pi_1\etc \pi_k)$ where $\pi_i$ is a partition of the vertex set of $C_i$. Moreover, for any $\pi$ such that $G^\pi_X=G$, $\theta(\pi)$ is completely determined by $G$ and shall be denoted $\theta(G)$. We hence get, as for the opposite case,
	\eq
		\sum_{ \substack{\pi \in  \mcal {DU}^{par}_{2\ell} \\ \mrm{s.t. \ } G^\pi_X=G} }\theta(\pi) \prod_{ C \in C_A} \frac 1 N \Tr^0[C]  &=& \theta(G)\prod_{i=1}^k \sum_{\pi_i\in \mcal P(V_{C_i})}   \frac 1 N \Tr^0[C_i^\pi]\\
		&=&  \theta(G)\prod_{i=1}^k \frac 1 N \Tr[C_i].
	\qe

\begin{definition}\label{NC1par}We denote by $\mcal Ann^{par}$ the labeled graph consisting of the annulus formed by the outer cycle $T_1$ and the inner cycle $T_2$, both in in anticlockwise orientation. For any $G\in  \mcal G^{par}$, we set $\sigma_G$ the partition of the edge set of $\mcal Ann^{par}$ labeled by Wigner matrices such that two edges belong to a same block if and only if they are twined in $G$.
\end{definition}

As before, $\sigma_G$ is a non crossing  pairing  which completely determines $G$. We denote $\theta_\sigma = \theta(G)$ and get
	\eq
		\lefteqn{\sum_{ G \in  \mcal G^{par}_{2\ell}} \sum_{ \substack{\pi \in  \mcal {DU}^{par}_{2\ell} \\ \mrm{s.t. \ } G^\pi_X=G} } \theta(G)\prod_{ C \in C_A} \frac 1 N \Tr^0[C]}\\
		& \limN& \sum_{\substack{\sigma \in  NC_{2}^{(2\ell)}(m,n) \\
  \text{non-mixing}}}
\theta_\sigma \phi_{ K(\sigma)}(a_1 \etc a_{m}, a_{m+n}^t \etc a_{m+1}^t).
	\qe

Assume now that $G \in \mcal G^{par}_{2}$. All $\pi$ such that $G^\pi_X = G$ have same graph $K(G) := G^\pi_A$ which consists in a union of cycles $C_1\etc C_k$. There are partitions $\pi'$ such that the two double edges of the double cycle are identified to form a group of edges of multiplicity 4, for which $G_A^{\pi'}$ is a union of $k-2$ cycles $\tilde C_1\etc \tilde C_{k-2}$, where two graphs $\tilde C_i$ are obtained by identifying a vertex of cycle $C_p$ and a vertex of a cycle $C_q$. As in the opposite case, we have
	\eq
		\lefteqn{\sum_{ \substack{\pi \in  \mcal {DU}^{par}_{2} \\ \mrm{s.t. \ } G^\pi_X=G} }\theta(\pi) \prod_{ C \in C_A} \frac 1 N \Tr^0[C]}\\
		 &=& \prod_{i=1}^k \sum_{\pi_i\in \mcal P(V_{C_i})}   \frac 1 N \Tr^0[C_i^\pi]\\
		&=&  \theta(G)\big( \prod_{i=1}^k \frac 1 N \Tr[C_i] - \prod_{i=1}^{k-2} \frac 1 N \Tr[\tilde C_i] \big)\\
		& \limN &  \theta_\sigma\phi_{ K(\sigma)}(a_1 \etc a_{m}, a_{m+n}^t \etc a_{m+1}^t) -  \theta_\sigma\tilde \phi_{ K(\sigma)}(a_1 \etc a_{m+n}).
	\qe
Finally, we obtain
	\eq
		\tau^{(2)}_{con} &\limN & \sum_{\substack{\sigma \in   NC_{2}(m,n)\\
  \text{non-mixing}}} \phi_{{  K(\sigma)}}(a_1\etc a_{m+n}) \\
		& & \qquad+ \sum_{\substack{\sigma \in   NC^{(2)}_{2}(m,n)\\
 \text{non-mixing}}}   \big( \E[|x^{(\sigma)}_{12}|^4]-2 - \theta_\sigma\big)  \tilde \phi_{ K(\sigma)}(a_1 \etc a_{m+n})\\
		& & \qquad+ \sum_{\substack{\sigma \in   NC_{2}(m,n)\\
  \text{non-mixing}}} \theta_\sigma (\phi_{{  (t)}})_{{  K(\sigma)}}(a_1\etc a_{m+n}).
	\qe
  This is the result as claimed in \eqref{eq:formula-phi2} in Theorem \ref{MainTh4}. Note that 
$$ \big( \E[|x^{(\sigma)}_{12}|^4]-2 - \theta_\sigma\big)=k_4$$ 
and that the last term in \eqref{eq:formula-phi2} is not showing up in the present case, where $m$ and $n$ are both even.

\subsection{Computation of the covariance: odd moments}
We hereafter assume that $m$ and $n$ are odd. Hence, by a parity argument, there is no partition $\pi$ of $V$ such that $T^\pi$ is of 4-2 type, and when $T^\pi$ is of double unicyclic type, then the number of edges labeled by Wigner matrices in the double cycle of $\mcal {GDC}(T^\pi)$ is odd.

\begin{figure}[t]
\includegraphics{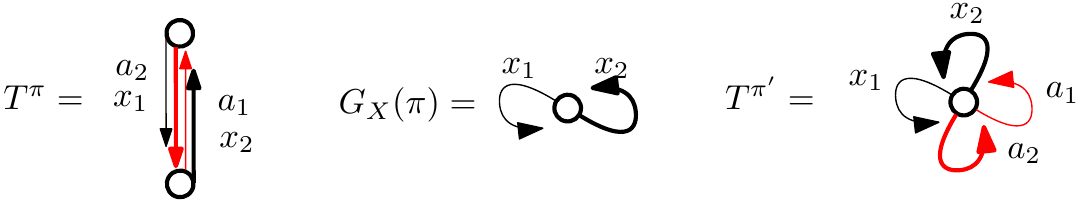}

\caption{Left: A labeled graph $T^\pi$ for $\pi \in   \mcal {DU}^{opp}_{1,+}$. Middle: the graph $G^\pi_X$. Right a labeled graph $T^{\pi'}$ for $\pi' \in   \mcal {DU}^{opp}_{1,-}$ such that $G^\pi_X = G_X^{\pi'}$.} 
\label{fig_19_OppPetit.pdf}
\end{figure}

We denote by $\mcal T^{(1)}$ the set of partitions $\pi$ of $V$ such that the double cycle of $\mcal {GDC}(T^\pi)$ is of length 1, i.e. it is a self-loop labeled by a Wigner matrix. In this situation,
	\eq
		\omega(\pi) = \E[ {x^{(\pi)}_{11}}^2],
	\qe
where $X^{(\pi)}=\big(\frac{x_{ij}^{(\pi)}}{\sqrt N}\big)$ denotes the Wigner matrix associated to the through string. If $\pi \in \mcal {DU}^{opp}\setminus \mcal T^{(1)}$, then $\omega(\pi) = 1$ and if $\pi \in \mcal {DU}^{par}\setminus \mcal T^{(1)}$ then $\omega(\pi) = \theta(\pi)$ defined in \eqref{Eq:W3} as before.

Let $G \in \mcal G^{par}_{2\ell+1} \cup \mcal G^{opp}_{2\ell+1}$ for $\ell\geq 1$. There is no modification of the reasoning, compare to the even moments case, and we get 
	\eq
		\lefteqn{\sum_{ G \in  \mcal G^{opp}_{2\ell+1}} \sum_{ \substack{\pi \in  \mcal {DU}^{opp}_{2\ell+1} \\ \mrm{s.t. \ } G^\pi_X=G} }\prod_{ C \in C_A} \frac 1 N \Tr^0[C]}\\
		& \limN& \sum_{ \sigma \in  NC^{(2\ell+1)}(m,n) }  \phi_{ K(\sigma)}(a_1 \etc a_{m+n})\\
	\qe
	\eq
		\lefteqn{\sum_{ G \in  \mcal G^{par}_{2\ell+1}} \sum_{ \substack{\pi \in  \mcal {DU}^{par}_{2\ell+1} \\ \mrm{s.t. \ } G^\pi_X=G} } \theta(G)\prod_{ C \in C_A} \frac 1 N \Tr^0[C]}\\
		& \limN& \sum_{ \sigma \in  NC^{(2\ell+1)}(m,n) } \theta_\sigma \phi_{ K(\sigma)}(a_1 \etc a_{m}, a_{m+n}^t \etc a_{m+1}^t).
	\qe

We denote by $\mcal {DU}^{opp}_{1,+}$ the set of $\pi \in  \mcal {DU}^{opp}_1$ (opposite type partitions with a single Wigner matrix on the double cycle) but the double cycle is not of length one, and $\mcal {DU}^{opp}_{1,-} = \mcal {DU}^{opp}_{1} \setminus \mcal {DU}^{opp}_{1,+}$. For $\pi \in \mcal {DU}^{opp}_{1,+}$, the graph $G^\pi_X$ is degenerated in the sense that it as a double self-loop. Given $G \in \mcal G^{opp}_1$, all graphs $\pi \in \mcal {DU}^{opp}_{1,+}$ such that $G^\pi_X = G$ have same graph $G^\pi_A$, which is a disjoint union of simple cycles $C_1\etc C_k$. Similarly all graphs $\pi' \in \mcal {DU}^{opp}_{1,-}$ such that $G^\pi_X = G$ have same graph $G_A^{\pi'}$, which is a disjoint union of simple cycles $\tilde C_1\etc \tilde C_k$ where each $\tilde C_i$ is equal to $C_i$ but one $\tilde C_i$, obtained by identifying vertices of $C_i$. Hence we get

	\eq
		\tau^{(2)}_{con} & \limN & \sum_{\substack{\sigma \in   NC_{2}(m,n)\\
  \text{non-mixing}}} \phi_{K(\sigma)}(a_1\etc a_{m+n}) \\
		& & + \sum_{\substack{\sigma \in   NC_{2}(m,n)\\
  \text{non-mixing}}}\theta_\sigma (\phi_{(t)})_{K(\sigma)}(a_1\etc a_{m+n}) \\
		& & + \sum_{\substack{\sigma \in   NC_2^{(1)}(m,n)\\
  \text{non-mixing}
}} (\eta_\sigma -1 - \theta_\sigma)  \tilde \phi_{K(\sigma)}(a_1\etc a_{m+n}).
	\qe

 This is the formula \eqref{eq:formula-phi2} as claimed in Theorem \ref{MainTh4}, if we take also into account that the term from \eqref{eq:formula-phi2} involving $k_4$ is not showing up in the present case where both $m$ and $n$ are odd.

This finishes the proof of both Theorem \ref{MainTh2} and Theorem \ref{MainTh4}.

\newpage
	
\thebottomline

\listcomments

 \end{document}